\documentclass[12pt]{amsart}
\usepackage{textcomp}
\usepackage{mathptmx}
\usepackage{anyfontsize}
\usepackage{t1enc}
\usepackage[top=2cm,bottom=2cm,left=2cm,right=2cm]{geometry}
\usepackage{setspace}
\usepackage{cite}
\usepackage{enumerate}
\usepackage{enumitem} 
\usepackage{amssymb}
\usepackage{dsfont}
\usepackage[dvips]{graphicx}
\usepackage{epsfig}
\usepackage{subfigure}
\usepackage{multicol}

\usepackage[utf8]{inputenc}
\usepackage[T1]{fontenc}
\usepackage{pdfrender,xcolor}

\newcommand{\g}[1]{\gamma_{#1}}

\newcommand{\Dset}{\mathop{C_{\scriptscriptstyle{\!\Rset\to \Cset}}^{\infty}}}
\newcommand{\mlim}{\displaystyle{\scriptstyle{\mom}}\text{-}\!\lim_{\hspace*{-5mm}h\to 0}}
\newcommand{\Cset}{\mathop{\mathds{C}}}
\newcommand{\Nset}{\mathop{\mathds{N}}}

\newcommand{\Rset}{\mathop{\mathds{R}}}
\newcommand{\sgn}{\mathop{\mathrm{sgn}}}
\newcommand{\st}[1]{\mathop{\mathds{S}(\frac{1}{#1\rho})}}
\newcommand{\sech}{\mathop{\mathrm{sech}}}
\newcommand{\mom}{\ensuremath{{\mathcal{M}}}}
\newcommand{\Mi}{~_{\!\!\!{\mathcal{M}}}}
\newcommand{\Mscal}[2]{\langle #1, #2\rangle_{~_{\!\!\!\!{\mathcal{M}}}}}
\newcommand{\K}[1]{\ensuremath{{\mathcal{K}}^{#1}}}
\newcommand{\LL}{\ensuremath{L_{\!{\scriptscriptstyle{\mathcal{M}}}}^2}}
\newcommand{\LB}{\ensuremath{{\mathbf L}_{{\!{\scriptscriptstyle{\mathcal{M}}}}}^2}}
\newcommand{\LT}{\ensuremath{L^2_{d\alpha} }}

\newcommand{\norm}[1]{\left\|#1\right\|_{\scriptscriptstyle{\!{\mathcal{M}}}}}
\newcommand{\noi}[1]{\left\| #1\right\|_{d\alpha}}
\newcommand{\doti}[2]{\langle #1, #2\rangle_{d\alpha}}

\newcommand{\KK}{\ensuremath{\{\K{n}_z\}^{\scriptscriptstyle{\mathcal{M}}}_{n\in \Nset}}}
\newcommand{\PP}{\ensuremath{\{p_{n}(\omega)\}_{n\in \Nset}}}
\newcommand{\sinc}[1]{\mathrm{sinc}\, #1}
\newcommand{\BL}[1]{\ensuremath{\mathbf{BL}(#1)}}
\newcommand{\CE}{\mathrm{CE}^{\!\scriptscriptstyle{\mathcal{M}}}}
\newcommand{\CA}{\mathrm{CA}^{\!\scriptscriptstyle{\mathcal{M}}}}
\newcommand{\mm}{\mathbi{m}}
\newcommand{\dd}{{\mathop{{D}}}}
\newcommand{\da}{{\mathop{{d}\alpha(\omega)}}}
\newcommand{\e}{\mathop{{\mathrm{e}}}}
\newcommand{\s}[1]{\mathop{{\mathrm{s}(#1)}}}
\newcommand{\ds}[1]{\mathop{{\mathrm{ds}(#1)}}}
\newcommand{\ii}{\mathop{\dot{\imath}}}
\newcommand{\wght}{\mathop{\mathrm{w}}}
\def\mathbi#1{\textbf{\textit #1}}
\newcommand{\CC}{\ensuremath{\mathcal{B}}}
\newcommand{\Norm}[1]{\left\|#1\right \|^{\scriptscriptstyle{{\mathcal{M}}}}}
\newcommand{\LIM}[1]{\lim_{#1\to \infty}}
\newcommand{\PT}[2]{p_{#1}^{\scriptscriptstyle{T}}(#2)}

\newcommand{\Langle}{\langle\hspace*{-1mm}\langle}
\newcommand{\Rangle}{\rangle\hspace*{-1mm}\rangle}
\newcommand{\FT}{\mathcal{F}_{\!\!\!\!\!\Mi}}

\newtheorem{theorem}{Theorem}[section]
\newtheorem{corollary}[theorem]{Corollary}
\newtheorem{proposition}[theorem]{Proposition}
\newtheorem{conjecture}[theorem]{Conjecture}
\newtheorem{lemma}[theorem]{Lemma}
\newtheorem{definition}[theorem]{Definition}

\title{Chromatic Derivatives, Chromatic Expansions and Associated Spaces II}
\date{}
\author{Aleksandar Ignjatovi\'{c}}

\address{School of Computer Science and Engineering,
University of New South Wales Sydney, NSW 2052, Australia}
\email{aleksignjat@gmail.com}

\keywords{orthogonal polynomials, special functions, series expansions, almost periodic functions, Fourier transforms, chromatic derivatives, chromatic expansions, signal representation, sampling theorems, signal processing}

\subjclass[2000]{41A58, 42C15, 94A12, 94A20}

\begin{document}

\begin{abstract}This paper extends and corrects some of the results from our previous paper   ``\emph{Chromatic Derivatives, Chromatic Expansions and Associated Spaces}" which appeared in the \emph{East Journal on Approximations}, vol 15, pp. 263-302 (2009).
We study differential operators associated with families of polynomials orthonormal with respect to certain measures. These operators, when applied to the Fourier transforms of such measures, produce basis functions for expansions of functions analytic on some complex domains. For many classical families of orthogonal polynomials these basis functions are the familiar special functions, such as the Bessel and the spherical Bessel functions.  Many familiar identities involving such special functions turn out to be just special cases of such expansions. We also use these differential operators to introduce some new spaces of almost periodic functions. The notions we study here have been successfully applied to signal processing, for example to recovery of band-limited signals from their non-uniform samples as well as from their zero crossings and the locations of their extremal points.
\end{abstract}

\maketitle

\section{Motivation}

Signal processing mostly deals with band-limited signals of finite energy, represented by
real valued continuous $L^2$ functions whose Fourier transform is supported within a finite interval. 
By choosing the unit time interval appropriately, such support interval is usually taken to be  
$[-\pi,\pi]$. The Hilbert space of such functions with the usual scalar product is denoted by $\BL{\pi}$.  

Foundations of the classical digital signal processing
rest on the Whittaker--Kotel'nikov--Nyquist--Shannon Sampling
Theorem (for brevity the Shannon Theorem):
every signal $f\in\BL{\pi}$ can be represented using its samples
at integers and \textit{the cardinal
sine function} $\sinc t =\sin \pi t/(\pi t)$, as
\begin{equation}\label{nexp}
f(t)=\sum_{n=-\infty}^{\infty}f(n)\; \sinc (t-n).
\end{equation}

Such signal representation is of \textit{global nature},
because it involves samples of the signal at integers of
arbitrarily large absolute value. In fact, since for a fixed
$t$ the values of $\sinc (t-n)$ decrease only as ${1}/{(\pi |n|)}$ as $|n|$ grows,
the truncations of the above series do not provide satisfactory
local signal approximations: to guarantee a small error of 
approximation around $t=0$ one would need a large number of distant samples.

On the other hand, since every signal $f\in\BL{\pi}$ is a restriction to $\Rset$
of an entire function, such a signal can also be represented by the Taylor series,
\begin{equation}\label{taylor}
f(t)=\sum_{n=0}^{\infty}f^{(n)}(0)\;\frac{t^n}{n!}.
\end{equation}
Such a series converges uniformly on every finite interval, and
its truncations provide good local signal approximations. Since
the values of the derivatives $f^{(n)}(0)$ are determined by
the values of the signal in an arbitrarily small neighborhood
of zero, the Taylor expansion is of \textit{local nature}. Thus, in
this sense, the Shannon and the Taylor expansions are of
complementary nature.

However, unlike the Shannon expansion, the Taylor expansion has
found very limited use in signal processing, due to several
problems associated with its application to empirically sampled
signals.
\begin{enumerate}
\item[(I)] Numerical evaluation of higher order derivatives
    of a function given by its samples is very noise
    sensitive. Generally, one is cautioned against numerical
    differentiation:
    \begin{quote}``\ldots
    numerical differentiation should be avoided whenever
    possible, particularly when the data are empirical and
    subject to appreciable errors of
    observation''\cite{Hil}.\end{quote}

\item[(II)] The Taylor expansion of  a signal $f\in\BL{\pi}$
    converges non-uniformly on $\Rset$; its truncations
    have rapid error accumulation when moving away from the center
    of expansion and are unbounded.

\item[(III)] Since the Shannon expansion of a signal $f\in\BL{\pi}$
    converges to $f$ in $\BL{\pi}$, the action of a
    continuous linear shift invariant operator (in signal
    processing terminology, a \textit{filter})  $F$
    can be expressed using samples of
    $f$ and the \textit{impulse response}\footnote{The impulse response
    of a filter $F$ is usually defined as the image $F[\delta]$ of the Dirac's $\delta$ distribution, but since we are interested in band limited signals only, if the range of $F$ is contained in $\BL{\pi}$, then such impulse response is equal to $F[\sinc(\pi t)]$.} of
    $F$:
\begin{equation}\label{filter-ny}
F[f](t)=\sum_{n=-\infty}^{\infty}f(n)\;F[\sinc\!](t-n).
\end{equation}
In contrast, the polynomials obtained by truncating the
Taylor series do not belong to $\BL{\pi}$ and nothing
similar to \eqref{filter-ny} is true of the Taylor
expansion.
\end{enumerate}

Chromatic derivatives were introduced in \cite{IG0} to overcome
problem (I) above; the chromatic approximations were introduced
in \cite{IG00} to obtain local approximations of band-limited
signals which do not suffer from problems (II) and (III).

\subsection{Numerical differentiation of band limited signals}

To understand the problem of numerical differentiation of
band-limited signals, we consider an arbitrary $f\in \BL{\pi}$
and its Fourier transform $\widehat{f}(\omega)$; then
\[f^{(n)}(t)=\frac{1}{2\pi}\int_{-\pi}^{\pi}(\ii\omega)^n
\widehat{f}(\omega){\e}^{\ii\omega t}d\omega.\]
Figure~\ref{derivatives} (left) shows, for $n= 15$ to $n=18$,
the plots of $(\omega/\pi)^{n}$, which are, save a factor of
$\ii^n$, the \emph{symbols}, or, in signal processing terminology, the
\textit{transfer functions} of the normalized derivatives
$1/\pi^n \; {\mathrm d}^n/{\mathrm d }t^n$. These plots reveal
why there can be no practical method for any reasonable
approximation of derivatives of higher orders. Multiplication
of the Fourier transform of a signal by the transfer function
of a normalised derivative of higher order obliterates the
Fourier transform of the signal, leaving only its edges, which
in practice contain only noise. Moreover, the graphs of the
transfer functions of the normalised derivatives of high orders
and of the same parity cluster so tightly together that they
are allmost indistinguishable; see Figure~\ref{derivatives}
(left).\footnote{If the derivatives are not normalised, their
values can be very large and are again determined essentially
by the noise present at the edge of the bandwidth.}

However, contrary to a common belief, these facts \textit{do
not} preclude numerical evaluation of all differential
operators of higher orders, but only indicate that, from a
numerical perspective, the set of the derivatives $\{f,
f^\prime, f^{\prime\prime},\ldots\}$ is a very poor basis of the
vector space of linear differential operators with constant 
coefficients. We now show how to obtain a basis for this space
consisting of numerically robust linear differential operators.

\subsection{Chromatic derivatives}
Let polynomials $p^{\scriptscriptstyle{L}}_n(\omega)$ be
obtained by normalizing and scaling the Legendre polynomials,
so that
\begin{equation*}
\frac{1}{2\pi}\int_{-\pi}^{\pi}p^{\scriptscriptstyle{L}}_{\scriptstyle{n}}
(\omega)\;p^{\scriptscriptstyle{L}}_{\scriptstyle{m}}(\omega){\rm d}\omega=\delta(m-n)
\end{equation*}
where $\delta(n)$ is the Kroneker function $\delta(0)=1$ and $\delta(n)=0$ for all $n\neq 0$. We define operator polynomials\footnote{Thus, obtaining $\K{n}_t$
involves replacing $\omega^k$ in
$p^{\scriptscriptstyle{L}}_{\scriptstyle{n}}(\omega)$ with
$(-\ii)^k\,{\rm d}^k/{\rm d} t^k$ for all
$k\leq n$. If $\K{n}_t$ is applied to a function of a single
variable, we drop index $t$ in $\K{n}_t$.}
which we call \emph{the chromatic derivatives associated with the Legendre polynomials}, 
\begin{equation*}\label{dop} \K{n}_t={\ii}^{n}
p^{\scriptscriptstyle{L}}_{\scriptstyle{n}}\left(-\ii\;\frac{{\rm d}}{{\rm d} t}\right).
\end{equation*}
Since polynomials
$p^{\scriptscriptstyle{L}}_{\scriptstyle{n}}(\omega)$ contain
only powers of the same parity as $n$, operators $\K{n}$ have
real coefficients, and it is easy to verify that 
\begin{equation*}
\K{n}_t[{\e}^{\ii\omega t}]=
{\ii}^np^{\scriptscriptstyle{L}}_{\scriptstyle{n}}(\omega)\,{\e}^{\ii\omega t}.
\end{equation*}
In fact, chromatic derivatives were defined precisely with the purpose of having this property; this way, for all $f\in\BL{\pi}$, since $\widehat{f}(\omega)\;{\e}^{\ii\omega t}$ is bounded and continuous in $\omega$ and $t$ on $(-\pi,\pi)\times (-\infty, \infty)$ and continuously differentiable with respect to $t$, we have
\[
\K{n}[f](t)=\frac{1}{2\pi}\int_{-\pi}^{\pi}{\ii}^n
p^{\scriptscriptstyle{L}}_{\scriptstyle{n}}(\omega)\widehat{f}(\omega)\;
{\e}^{\ii\omega t}{\rm d}\omega.
\]

Figure~\ref{derivatives} (right)
shows the plots of
$p^{\scriptscriptstyle{L}}_{\scriptstyle{n}}(\omega)$, for $n=
15$ to $n=18$, which are the transfer functions (again save a
factor of $\ii^n$) of the corresponding operators $\K{n}$.
Unlike the transfer functions of the (normalized) derivatives
$1/\pi^n\;{\rm d}^n/{\rm d} t^n$ shown on the left, the transfer functions of the
chromatic derivatives $\K{n}$ form a family of well separated,
interleaved and increasingly refined comb filters. Instead of
obliterating, such operators encode the features of the Fourier
transform of the signal (in signal processing terminology, the
\textit{spectral features} of the signal). For this reason, we
call operators $\K{n}$ the {\em chromatic derivatives\/}
associated with the Legendre polynomials.

\begin{figure}
    \begin{center}
     \includegraphics[width=4.7in]{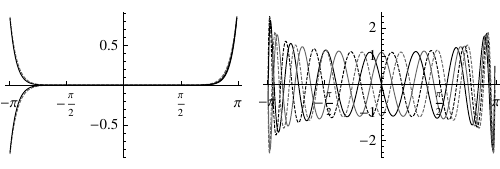}
      \caption{Graphs of $\left(\frac{\omega}{\pi}\right)^n$ (left) and of
      $P^{\scriptscriptstyle{L}}_{\scriptstyle{n}}(\omega)$
      (right) for $n=15-18$.}
      \label{derivatives}
    \end{center}
\end{figure}
Chromatic derivatives can be accurately and robustly evaluated from samples of the signal taken at the usual rate, slightly higher than the Nyquist rate, so that the useful bandwidth of the signal is contained within interval $[-.9\pi, .9\pi]$, thus also providing a slight guard band at the edges of the bandwidth. On Figure~\ref{remezLegendre32} the plot on the left shows the transfer function $\left(\frac{\omega}{\pi}\right)^{32}$ of the normalised standard derivative $\frac{1}{\pi^{32}}\frac{d^{32}}{dt^{32}}$ of order $32$ (gray) and the transfer function of its 129 tap FIR approximation\footnote{FIR filters or Finite Impulse Response filters with $2N+1$ coefficients  $c_k$ (in signal processing terminology \emph{taps}) are given by $\mathcal{A}[f](t)=\sum_{k=-N}^{N}c_{k}\,f(t+k)$.}  (black); the plot in the middle represents the transfer function of the ideal filter corresponding to the chromatic derivative $\K{32}$ of order 32 (gray) and the transfer function of its 129 tap FIR approximation (black). The
pass-band of both implemented filters is 90\% of the full bandwidth $[-\pi,\pi]$; the stop-band is $|\omega|\geq 0.98\pi$; both filters were designed using the Remez exchange algorithm \cite{Opp}. The plot on the right shows the error of a 129 tap FIR approximation of $\K{32}$ over its bandpass.  As it can clearly be seen on that figure, the FIR filter approximating the normalised standard derivative is essentially useless; on the other hand the FIR approximation of the chromatic derivative of the same order has an error of less than $10^{-6}$ within the entire signal bandwidth $|\omega|\leq .9\pi$.
Implementations of filters for
operators $\K{n}$ of orders higher than 100 have been tested
in practice and proved to be both extremely accurate and noise robust given sufficient (but still reasonable) number of taps (i.e., samples of the signal); see \cite{IWK1} and \cite{IWK2}.
\begin{figure}
    \begin{center}
      \includegraphics[width=7.2in]{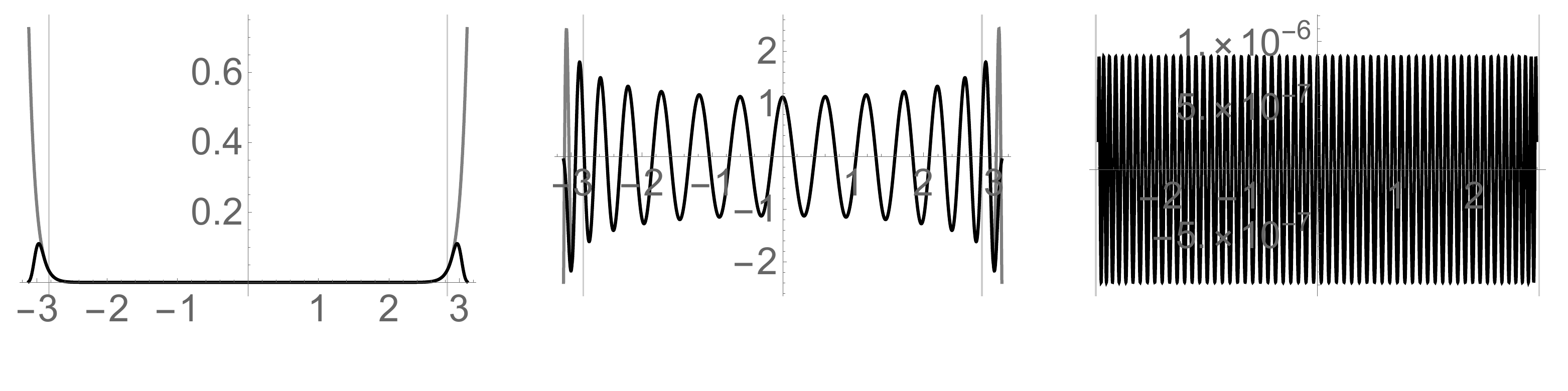}
      \caption{Left: transfer function of the (normalised) standard derivative $\frac{1}{\pi^{32}}\frac{d^{32}}{dt^{32}}$ of order 32 (gray) and of its FIR filter approximation (black). Center: transfer function of the chromatic derivative $\K{32}$ (gray) and of its FIR approximation (black). Right: the error of a 129 tap FIR filter approximation for $\K{32}$ within its bandpass $[-.9\pi,.9\pi]$.}
      \label{remezLegendre32}
    \end{center}
\end{figure}

In particular, one can show that
\begin{equation}\label{cdersinc}
\K{n}[\sinc](t)=(-1)^n\;\sqrt{2n+1}\;{\mathrm j}_{n}(\pi t),
\end{equation}
where ${\mathrm j}_n(x)$ is the spherical Bessel function of the first kind of order $n$. 
Note that \eqref{nexp} and \eqref{cdersinc} imply that
\begin{equation}\label{no}
\K{k}[f](t)=\sum_{n=-\infty}^{\infty}f(n)\;\K{k}[\sinc\!](t-n)
=\sum_{n=-\infty}^{\infty}f(n)\;(-1)^k\;\sqrt{2k+1}\;{\mathrm j}_k(\pi(t-n)).
\end{equation}
However, in practice, the values of $\K{k}[f](t)$, especially
for larger values of $k$, \textit{cannot} be obtained from the
Nyquist rate samples by differentiating the Shannon expansion 
as in \eqref{no}. This is
due to the fact that functions $\K{k}[\sinc\!](t-n)$ decay very
slowly as $|n|$ grows; see Figure~\ref{j15} (left). Thus, to
achieve any accuracy, a truncated Shannon expansion would 
have to contain an extremely large number of terms.

\begin{figure}
    \begin{center}
      \includegraphics[width=6.2in]{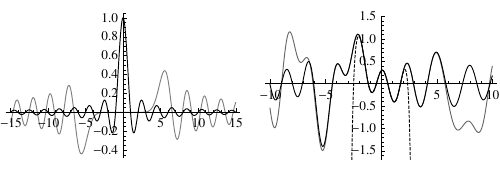}
      \caption{(left) plots of $\sinc(t)$ (black) and $\K{15}[\sinc](t)$ (grey); (right) A band-limited signal (black) and its chromatic approximation of order 15 (grey) and its Taylor approximation of the same order (dashed).}
      \label{j15}
    \end{center}
\end{figure}

\subsection{Chromatic expansions}
The above considerations show that numerical evaluation of 
the chromatic
derivatives associated with the Legendre polynomials does not
suffer problems which preclude numerical evaluation of the
``standard'' derivatives of higher orders. On the other hand,
the chromatic expansions, defined in Proposition~\ref{app}
below, were conceived as a solution to the problems associated with
the use of the Taylor expansion.\footnote{Propositions stated
in this section are special cases of general propositions
proved in subsequent sections.}

\begin{proposition}\label{app}
Let $\K{n}$ be the chromatic derivatives associated with the
Legendre polynomials, let ${\mathrm j}_n(x)$ be the spherical
Bessel function of the first kind of order $n$, and let $f$  be
an arbitrary entire function; then for all $z,u\in\Cset$,
\begin{eqnarray}
f(z)&=&\sum_{n=0}^{\infty}\;\K{n}[f](u)\;  \K{n}_u[\sinc (z-u)]\label{CEL}\\
&=&\sum_{n=0}^{\infty}(-1)^n\;\K{n}[f](u)\;  \K{n}[\sinc ](z-u)\label{CEL1}\\
&=&\sum_{n=0}^{\infty}\K{n}[f](u)\;\sqrt{2n+1}\;  {\mathrm j}_n (\pi(z-u))
\end{eqnarray}
Moreover, if  $f\in\BL{\pi}$, then the series converges uniformly on
$\Rset$ and also converges with respect to the $L^2$ norm of the space $\BL{\pi}$.\footnote{Note that $\sinc(t)={\mathrm j}_0(\pi t)$.}
\end{proposition}

The series in \eqref{CEL} is called \textit{the chromatic
expansion of $f$ associated with the Legendre polynomials\/}; a
truncation of this series is called a \textit{chromatic
approximation\/} of $f$. As the Taylor approximation, a
chromatic approximation is also a local approximation; its
coefficients are the values of differential operators
$\K{m}[f](u)$ at a single instant $u$, and for all $k\leq n$,
\begin{equation*}
f^{(k)}(u)=\frac{{\rm d}^k}{{\rm d}t^k}
\left[\sum_{m=0}^{n}\K{m}[f](u) \;\K{m}_u[\sinc(t-u)]\right]_{t=u}.
\end{equation*}

Figure~\ref{j15} (right) compares the behavior of the chromatic
approximation (black) of a signal $f\in\BL{\pi}$ (gray) with
the behavior of its Taylor approximation (dashed).
Both approximations are of order 15. The signal $f(t)$
is defined using the Shannon expansion, with
samples $\{f(n)\ : \ |f(n)|<1,\  -32\leq n\leq 32\}$ which were
randomly generated. The plot
reveals that, when approximating a signal $f\in\BL{\pi}$, a
chromatic approximation has a much gentler error accumulation
when moving away from the point of expansion than the Taylor
approximation of the same order. Moreover, unlike the monomials 
which appear in the Taylor formula,
functions $\sqrt{2n+1}\;{\mathrm
j}_n(\pi t)$ belong to $\BL{\pi}$ and satisfy
$|\K{n}[\sinc\!](t)|\leq 1$ for all $t\in\Rset$. Consequently,
the chromatic approximations obtained by truncating the chromatic expansions
also belong to $\BL{\pi}$ and are
bounded over the entire set of reals $\Rset$.

Since by Proposition~\ref{app} the chromatic approximation of a
signal $f\in\BL{\pi}$ converges to $f$ in $\BL{\pi}$, if $A$ is
a filter, then $A$ commutes with the differential operators
$\K{n}$ and thus for every $f\in \BL{\pi}$,
\begin{equation}\label{filter-ce}
A[f](t)=\sum_{n=0}^{\infty}(-1)^n\;\K{n}[f](u)\;
\K{n}[A[\,\sinc ]](t-u).
\end{equation}
A comparison of \eqref{filter-ce} with \eqref{filter-ny}
provides further evidence that, while local just like the
Taylor expansion, the chromatic expansion associated with the
Legendre polynomials possesses the features that make the
Shannon expansion so useful in signal processing. This,
together with numerical robustness of chromatic derivatives,
makes chromatic approximations applicable in fields involving
empirically sampled data, such as digital signal and image
processing; see e.g. \cite{WAKI}, \cite{WIK}, \cite{WSI}, \cite{IWK2}, \cite{IGF}, \cite{Sav}.

\subsection{A local definition of the scalar product in \BL{\pi}}
Proposition~\ref{fun} below demonstrates another important 
property of the chromatic derivatives associated with
the Legendre polynomials.

\begin{proposition}\label{fun}
Let $f:\Rset\rightarrow\Rset$ be a restriction of an
entire function; then the following are equivalent:
\begin{enumerate}
\item $\sum_{n=0}^{\infty}\K{n}[f](0)^2 < \infty$;
\item for all $t\in\Rset$ the sum $\sum_{n=0}^{\infty}\K{n}[f](t)^2$
converges to a constant function, i.e., its values are independent of $t\in\Rset$;
\item $f\in\BL{\pi}$.
\end{enumerate}
\end{proposition}

The next proposition is relevant for signal
processing because it provides \textit{local representations}
of the usual norm, the scalar product and the convolution in $\BL{\pi}$,
respectively, which are defined globally, as improper integrals.

\begin{proposition}\label{local-space} Let $\K{n}$ be the chromatic
derivatives associated with the (rescaled and normalized)
Legendre polynomials, and $f, g\in\BL{\pi}$. Then the following
sums do not depend on $t\in\Rset$ and satisfy
\begin{eqnarray}
\sum_{n=0}^{\infty}\K{n}[f](t)^2\hspace*{-2mm}&=&\hspace*{-2mm}
\int_{-\infty}^{\infty} f(x)^2 dx;\\
\sum_{n=0}^{\infty}\K{n}[f](t)\K{n}[g](t)
\hspace*{-2mm}&=&\hspace*{-2mm} \int_{-\infty}^{\infty}
f(x)g(x) dx;\\
\sum_{n=0}^{\infty}\K{n}[f](t) \K{n}_t[g(u-t)]
\hspace*{-2mm}&=&\hspace*{-2mm}
\int_{-\infty}^{\infty}f(x)g(u-x) dx.
\end{eqnarray}
\end{proposition}


In this paper we consider chromatic derivatives and chromatic
expansions which correspond to some very general families of
orthogonal polynomials, and prove generalisations of the above
propositions, also extending our previous work.

\section{Basic Notions}
\subsection{Families of orthogonal polynomials} Let
$\mom:\mathcal{P}_{\omega}\to  \Rset$ be a linear
functional on the vector space $\mathcal{P}_{\omega}$ of
polynomials with real coefficients in the variable $\omega$ and let
${\mu}_{n} = \mom(\omega^{n})$. Such $\mom$ is called a \textit{moment
functional} and ${\mu}_{n}$ is the \textit{moment} of
$\mom$ of order $n$.  The \textit{Hankel determinant} of order $n$ is defined as 
$\Delta_n=\left|\mu_{i+j}\right|_{i,j=0}^n$.
The moment functionals $\mom$ which we consider are assumed to be positive definite, i.e., $\Delta_{n}>0$ for all $n$;
such functionals also satisfy $\mu_{2n}> 0$. We also assume that \mom\ is normalised, so that $\mom(1)=\mu_0=1$.
For such functionals there exists a non-decreasing bounded function $\alpha(\omega)$ taking infinitely many values, called a \emph{moment distribution function},  such that the following improper Stieltjes integrals exist and satisfy
\begin{equation}\label{l3}
\mu_n=\int_{-\infty}^{\infty}
\omega^{n}\,\da.
\end{equation}
\begin{definition}
A moment functional \mom\ and the corresponding moment distribution function $\alpha(\omega)$ are called \emph{chromatic}
if the corresponding moments  $\mu_n$ satisfy 
\begin{equation}\label{limsup}
\rho=\limsup_{n\to \infty}
\left(\frac{\mu_{2n}}{(2n)!}\right)^{\frac{1}{2n}}={e}\,
\limsup_{n\to \infty} \frac{\mu_{2n}^{\frac{1}{2n}}}{2n}<\infty.
\end{equation}
\end{definition}

Let $\PP$ be the  family of orthonormal polynomials corresponding to \mom,  such that $p_{n}(\omega)$ is of degree $n$ and
\begin{equation}\label{polyortho}
\mom(p_{n}(\omega)\,p_{m}(\omega))=\int_{-\infty}^{\infty }p_{n}(\omega)\,p_{m}(\omega)
\,\da=\delta(m-n).
\end{equation}
Note that $p_0(\omega)\equiv 1$ because, by normalisation, $\mu_0=1$.
For such a family of polynomials  there exists a sequence of positive reals
$\gamma_n>0$ and a sequence of reals $\beta_n$ such that for all integers $n\geq 1$,
\begin{equation}\label{poly}
p_{n+1}(\omega)=\frac{\omega+\beta_n}{\gamma_{n}}\,
p_{n}(\omega)-\frac{\gamma_{n-1}}{\gamma_{n}}\,
p_{n-1}(\omega).
\end{equation}
 If we set $\gamma_{-1}=1$ and $p_{-1}(\omega)\equiv 0$,
then \eqref{poly} holds for $n=0$ as well. If $\mu_{2n+1}=0$ for all $n$, then such \mom\  is \emph{symmetric}, in which case $\beta_n=0$ for all $n$ and each polynomial $p_{n}(\omega)$ contains only powers of $\omega$ of the same parity as $n$.

We will make use of the Christoffel-Darboux equality for orthonormal polynomials,
\begin{equation}\label{CDP}
\sum_{k=0}^{n}
p_{k}(\omega)\,p_{k}(\sigma)
= \frac{\gamma_n
(p_{n+1}(\omega)\,p_{n}(\sigma)-p_{n+1}(\sigma)\,p_{n}(\omega))}{\omega-\sigma}.
\end{equation}
By subtracting and adding $p_{n+1}(\sigma)\,p_{n}(\sigma)$ within the parentheses on the right of   \eqref{CDP}  and then letting $\sigma\to \omega$ we obtain
\begin{equation}\label{sumsquares}
\sum_{k=0}^{n}p_{k}(\omega)^2 =  \gamma_n \,
(p^\prime_{n+1}(\omega)\,
p_{n}(\omega)-p_{n+1}(\omega)\,
p_{n}^\prime(\omega)).
\end{equation}

\subsection{Space $\LT$.}
We denote by $\LT$  the Hilbert space of functions\footnote{We slightly abuse the notation here, because we allow $\varphi(\omega)$ to contain fixed complex parameters.} $\varphi:\Rset\to  \Cset$ for which the improper Lebesgue - Stieltjes integral
$\int_{-\infty}^{\infty}|\varphi(\omega)|^2\, \da$ exists and is finite, with the scalar product defined by
$\doti{\varphi}{\psi}=\int_{-\infty}^{\infty}\varphi(\omega) \,\overline{\psi(\omega)}\, \da$, and with the
corresponding norm denoted by $\noi{\varphi}$.

\begin{lemma}\label{meuw}
Let $\alpha(\omega)$ be a chromatic moment distribution function and $\rho$ such that
\eqref{limsup} holds. Then for every real $a$ such that
$0\leq a  < 1/\rho$,\footnote{We set $1/0=\infty$.}
\begin{equation}\label{weight}
\int_{-\infty}^{\infty}{\e}^{a |\omega|}\da <\infty.
\end{equation}
\end{lemma}
\begin{proof}
Since for all $b>0$ the series $\sum_{n=0}^{\infty}{a ^n|\omega|^n}/{n!}$ converges to ${\e}^{a |\omega|}$ uniformly on $[-b,b]$, we obtain 
$\int_{-b}^{b}{\e}^{a |\omega|}\da =
\sum_{n=0}^{\infty}{a ^n}/{n!}\int_{-b}^{b}|\omega|^n\da.$
For even $n$ we have $\int_{-b}^{b}\omega^n\da\leq\mu_n.$
For odd $n$ we have $|\omega|^n<1+\omega^{n+1}$ for
all $\omega$, and thus
$\int_{-b}^{b}|\omega|^n\da <\int_{-b}^{b}\da +
\int_{-b}^{b}\omega^{n+1}\da\leq 1 + \mu_{n+1}.$
Let $\eta_n=\mu_n$ if $n$ is even, and $\eta_n=1+\mu_{n+1}$
if $n$ is odd. Then also
$\int_{-\infty}^{\infty}{\e}^{a |\omega|}\da \leq
\sum_{n=0}^{\infty}{a ^n}\eta_n/{n!}$.
It is easy to see that  $\limsup_{n\to \infty}(\eta_n/n!)^{1/n}=\rho$; thus,
for  $0\leq a <1/\rho$ the last sum converges to a
finite limit.
\end{proof}

On the other hand, the proof of Theorem 5.2
in \S II.5 of \cite{Fr} shows that if
\eqref{weight} holds for some $a > 0$, then
\eqref{limsup} also holds for some $\rho< 1/a $.
Thus, we get the following Corollary.
\begin{corollary}
A moment distribution function $\alpha(\omega)$ is chromatic just in case it satisfies \eqref{weight} for some $a >0$.
\end{corollary}

\noindent {For the remaining part of this
paper we assume that $\alpha(\omega)$ is a chromatic moment distribution function.}

For every $0<c\leq \infty$, we let $\mathds{S}(c)=\{z\in
\Cset\;:\;|\mathfrak{Im}(z)|<c\}$.

\begin{proposition}\label{anaz}
Let for $z\in \st{}$
\begin{equation}\label{mmz}
\mm(z)=\int_{-\infty}^{\infty}{\e}^{{\ii}\,\omega z }\da.
\end{equation}
Then $\mm(z)$ is analytic on the strip $\st{}$ and for all integers $n\geq 0$,
\begin{equation}\label{dmmz}
\mm^{(n)}(z)=\int_{-\infty}^{\infty}({\ii}\,\omega)^n{\e}^{{\ii}\,\omega z }\da.
\end{equation}
\end{proposition}
\begin{proof}
Fix $n$ and let $z=x+\ii y$ with $|y|<1/\rho$; then for
every $b>0$,
\begin{equation*}
\int_{-b}^{b}|
(\ii \omega)^n{\e}^{{\ii}\omega z }|\da
=\int_{-b}^{b}|\omega|^n {\e}^{|\omega y|}\da
=\sum_{k=0}^{\infty}\frac{|y|^k}{k!}\int_{-b}^{b} |\omega|^{n+k}\da.
\end{equation*}
As in the proof of  Lemma \ref{meuw}, we let
$\eta_m=\mu_m$ for even $m$ and
$\eta_m=1+\mu_{m+1}$ for odd $m$; then the above
implies
\[\int_{-\infty}^{\infty}|
(\ii \omega)^n{\e}^{{\ii}\omega z }|\da
\leq\sum_{k=0}^{\infty}\frac{|y|^k\eta_{n+k}}{k!},\] and it is
easy to see that for every fixed $n$,
$\limsup_{k\to \infty} (\eta_{n+k}/k!)^{1/k}=\rho$. Thus, for $z\in \st{}$ we can differentiate under the integral in \eqref{mmz} any number of times to obtain \eqref{dmmz}.
\end{proof}

Since for all non-negative integers $n$
\begin{align}\label{dm}
\mm^{(n)}(0)={\ii}^n\mu_{n}
\end{align}
and since in the symmetric case all moments of odd order are zero,  for symmetric moment functionals function $\mm(t)$, which is the restriction of function $\mm(z)$ to the reals, is real valued.

\begin{lemma}\label{complete}
If \mom\ is chromatic, then $\PP$ is a complete system in
$\LT$.
\end{lemma}

\begin{proof}
Follows from a theorem of  Riesz (\cite{Riesz}; see, for example, Theorem 5.1
in \S II.5 of \cite{Fr}) which asserts that if
$\liminf_{n\to \infty}
{\mu_{2n}^{1/(2n)}}/(2n)<\infty$, then $\PP$ is a
complete system in $\LT$.\footnote{Note that we need the
stronger condition $\limsup_{n\to \infty}
\left({\mu_{2n}}/{(2n)!}\right)^{1/(2n)}<\infty$ to ensure that function
$\mm(z)$ defined by \eqref{mmz} is analytic on a strip
(Proposition~\ref{anaz}).}
\end{proof}

\subsection{Chromatic derivatives}

Let $D^n_z=d^n/dz^n$; with every family of orthonormal polynomials \PP\  we associate a family of linear
differential operators $\KK$ defined by the operator
polynomial\footnote{Thus, to obtain $\K{n}_z$, one replaces
$\omega^k$ in $p_{n}(\omega)$ by $(-\ii)^k D^k_z$, where
$\dd^k_z[f(z,\vec{w})]=\frac{\partial^k}{\partial z^k}f(z,\vec{w})$. We use
the square brackets to indicate the arguments of operators
acting on various function spaces. If $A$ is a differential
operator, and if a function $f(z,\vec{w})$ has parameters
$\vec{w}$, we write $A_{z}[f(z,\vec{w})]$ to distinguish the variable $z$
of differentiation; if $f(z)$ contains only variable $z$, we
write $A[f(z)]$ for $A_{z}[f(t)]$ and $\dd^k[f(z)]$ for
$\dd^k_z [f(z)]$. }
\begin{align*}
\K{n}_z={\ii}^n\, p_{n}
\left(-\ii \dd_z\right),
\end{align*}
and call them \textit{the chromatic derivatives} associated with
polynomials \PP.  Equation \eqref{poly} implies that such operators satisfy
the recurrence
\begin{equation}\label{three-term}
\K{n+1}=\frac{1}{{\gamma_{n}}}\,\left(\dd\circ
\K{n}+\ii \beta_n\K{n}+\gamma_{n-1}\, \K{n-1}\right)
\end{equation}
with the same coefficients $\gamma_n>0$ and $\beta_n$ as in \eqref{poly}.
The definition of the chromatic derivatives was chosen so that 
\begin{equation}\label{iwt}
\K{n}_z[e^{\ii\omega z}]=
{\ii}^n p_{n}(\omega)\,{\e}^{\ii\omega z}
\end{equation}
and consequently 
\begin{equation}
\K{n}_z[\mm](z)=\int_{-\infty}^{\infty}{\ii}^{n}\,p_{n}(\omega)\,{\e}^{{\ii}\omega z} \da.
\end{equation}
Note that for the symmetric families of orthonormal polynomials $\beta_n=0$ and thus the corresponding chromatic derivatives have real coefficients.  

The basic properties of orthonormal polynomials imply that for all integers $m,n$,
\begin{equation}\label{orthonorm}
(-1)^{n}(\K{n}\circ \K{m})[\mm](0)=\delta(m-n),
\end{equation}
and, if $m < n$ then
\begin{equation}\label{bkn}
(\dd^m\circ\K{n})[\mm](0)=0.
\end{equation}
For  $m<n$ all powers of $z$ in $\K{m}\left[{z^n}\right]$ are positive; thus, if $m<n$ then
\begin{equation}
\K{m}\left[{z^n}\right](0)=0. 
\label{zero-mon}\end{equation}

The following lemma corresponds to the Christoffel - Darboux equality for orthonormal polynomials and has a similar proof. 

\begin{lemma}[\!\!\cite{IG5}]\label{CD} For all infinitely differentiable functions $f(t),g(t) \,:\Rset \to  \Cset$ of a real variable $t$
\begin{equation}\label{C-D}
 \dd_t\left[\sum_{m=0}^{n} \K{m}[f(t)]\,\overline{\K{m}[g(t)]}\right]= {\gamma_{n}}\,
(\K{n+1}[f(t)]\,\overline{\K{n}[g(t)]}+\K{n}[f(t)]\,\overline{\K{n+1}[g(t)]}).
\end{equation}
\end{lemma}

\begin{proof} We  apply  \eqref{three-term}  to $f(t)$ and multiply both sides of the resulting equation by $\overline{\K{m}[g]}$ to obtain
\begin{align*}
\gamma_m\K{m+1}[f]\overline{\K{m}[g]}=(\dd_t\circ\K{m})[f]\; \overline{\K{m}[g]}+\ii\beta_m \K{m}[f]\overline{\K{m}[g]}+
\gamma_{m-1}\K{m-1}[f]\overline{\K{m}[g]}
\end{align*} 
We then apply \eqref{three-term} to $g(t)$, take the complex conjugate of both sides of the resulting equation and multiply both sides by $\K{m}[f]$ to obtain 
\begin{align*}
\gamma_m\overline{\K{m+1}[g]}{\K{m}[f]}=\overline{(\dd_t\circ\K{m})[g]}\; {\K{m}[f]}-\ii\beta_m \overline{\K{m}[g]}{\K{m}[f]}+
\gamma_{m-1}\overline{\K{m-1}[g]}{\K{m}[f]}
\end{align*} 
Summing  these two equations and rearranging  the terms we obtain 
\begin{align*}
&(\dd_t\circ\K{m})[f]\; \overline{\K{m}[g]}+\overline{(\dd_t\circ\K{m})[g]}\; {\K{m}[f]}=\\
&\hspace*{20mm}\gamma_m(\K{m+1}[f]\overline{\K{m}[g]}+\overline{\K{m+1}[g]}{\K{m}[f]})-\gamma_{m-1}(\K{m}[f]\overline{\K{m-1}[g]}+\overline{\K{m}[g]}{\K{m-1}[f]}).
\end{align*} 
Since $\overline{\dd_t[f](t)}=\dd_t \overline{f(t)}$ for functions of a real variable $t$, the lefthand side of this equation is equal to $\dd_t[\K{m}[f]\; \overline{\K{m}[g]}]$. 
Summing such equations for all $0\leq m\leq n$ the righthand sides form a telescoping sum equal to the righthand side of \eqref{C-D}.
\end{proof}

\subsection{Chromatic expansions} Let $f:\Rset\to \Rset$ or $f:\Cset\to \Cset$ be infinitely
differentiable at a real or complex $u$; the formal series
\begin{align}\label{cex}
\CE[f,u](z)&=\sum_{k=0}^{\infty}(-1)^k\K{k}[f](u)\;\K{k}[\mm](z-u)\nonumber
\end{align}
is called the \textit{chromatic expansion} of $f$  associated 
with \mom,\footnote{Sometimes we abuse the terminology and say that the expansion is associated with polynomials \PP\  or with the operators $\{\K{n}\}_{n\in\Nset}$.}   centred at $u$, and
\begin{equation*}
\CA[f,n,u](z)=\sum_{k=0}^{n}(-1)^k\K{k}[f](u)\K{k}[\mm](z-u)
\end{equation*}
is the \textit{chromatic approximation} of $f$ of order $n$. 
From \eqref{orthonorm} it follows that the chromatic
approximation $\CA[f,n,u](z)$ of order $n$ of $f(z)$ for all
$m\leq n$ satisfies
\begin{align*}
\K{m}_{\!z} [\CA[f,n,u](z)]\big |_{z=u}&=\sum_{k=0}^{n} \,(-1)^k\K{k}
[f](u)\, (\K{m}\circ\K{k})[\mm](0)\ =\ \K{m}[f](u).
\end{align*}

Since $\dd^m$ can be expressed as a linear combination of operators $\K{k}$
for $k\leq m$, also $f^{(m)}(u) = \dd^m_{\!z}[\CA[f,n,u](z)]\big |_{z=u}$
for all $m\leq n$. Thus, in particular,
\begin{align}\label{k-to-d}
f^{(n)}(u) = \dd^n_z[\CA[f,n,u](z)]\big |_{z=u} \!\!&=
\sum_{k=0}^{n}(-1)^k \, \K{k}[f](u)\,(\dd^{n}\circ \K{k})[\mm](0)
\end{align}

Similarly, since  for $m\leq n$
\[f^{(m)}(u)={\dd^m_{\!z}}\!\left[\sum_{k=0}^{n}f^{(k)}(u)\frac{(z-u)^k}{k!}\right]
_{z=u}\] also 
\begin{align}\label{d-to-k}
\K{n}[f](u)
=\K{n}_z\left[\sum_{k=0}^{n}f^{(k)}(u)\frac{(z-u)^k}{k!}\right]_{z=u}=\!\sum_{k=0}^{n}f^{(k)}(u)\,
\K{n}_z\left[\frac{z^k}{k!}\right]\!\!(0)
\end{align}

Equations \eqref{k-to-d} and \eqref{d-to-k} relate
the standard and the chromatic bases of the vector space of linear
differential operators with constant coefficients,
\begin{align}
\dd^{n} &= \sum_{k=0}^{n}(-1)^k \, (\dd^{n} \circ \K{k})[\mm](0)\;
\K{k};
\label{inverse}\\
\K{n}&=\sum_{k=0}^{n} \K{n}\left[\frac{z^k}{k!}\right]\!\!(0)\;
\dd^k.\label{direct}
\end{align}

\subsection{Space $\LB$}
\begin{definition}
$\LB$ is the vector space of functions $f:\st{2}
\to  \Cset$ which are analytic on $\st{2}$ and
satisfy $\sum_{n=0}^{\infty}|\K{n}[f](0)|^2<\infty$.
\end{definition}

\begin{definition}\label{bfL2}
Let $f(z)\in \LB$; then function 
\begin{equation}\label{ft0}
 \varphi_f(\omega)=\sum_{n=0}^{\infty}
(-{\ii})^{n}\K{n}[f](0)\,p_{n}(\omega)
\end{equation}
belongs to $\LT$ and we call it the Fourier transform of $f(z)$ with respect to $\mom$, denoted by
$\FT[f](\omega)$.
\end{definition}

\begin{proposition}\label{f-phi}
Let $\varphi(\omega)\in \LT$; we can define a
corresponding function
$f_{\varphi}:\st{2}\to  \Cset$ by
\begin{equation}\label{ift0}
f_{\varphi}(z)=\int_{-\infty}^{\infty}\varphi(\omega)
e^{{\ii}\omega z }\,\da.
\end{equation}
Such $f_{\varphi}(z)$ is analytic on $\st{2}$ and for all $n$
and all $z\in \st{2}$,
\begin{equation}\label{fourier-int-der}
\K{n}[f_{\varphi}](z)=\int_{-\infty}^{\infty}
{\ii}^{n}\,p_{n}(\omega)\,
\varphi(\omega)\, {\e}^{{\ii}\omega z} \da.
\end{equation}
\end{proposition}

\begin{proof}
Let $z = x +\ii y$, with $|y|<1/(2\rho)$. For every fixed $n$ and
$b>0$ we have
\begin{align*}
\hspace*{-20mm}\int_{-b}^{b}| (\ii \omega)^n
\varphi(\omega){\e}^{{\ii}\omega z}|\da
&=\int_{-b}^{b}|\omega|^n |\varphi(\omega)
|{\e}^{|\omega y|}\da\\
&=\sum_{k=0}^{\infty}\frac{|y|^k}{k!}\int_{-b}^{b}
|\omega|^{n+k}|\varphi(\omega)| \da\\
&\leq\sum_{k=0}^{\infty}\frac{|y|^k}{k!}
\left(\int_{-b}^{b} \omega^{2n+2k}\da\int_{-b}^{b}
|\varphi(\omega)|^2 \da\right)^{1/2}.
\end{align*}
Thus, also \[\int_{-\infty}^{\infty}| (\ii \omega)^n
\varphi(\omega){\e}^{{\ii}\omega z }|\da\leq
\noi{\varphi(\omega)}
\sum_{k=0}^{\infty}\frac{\sqrt{\mu_{2n+2k}}}{k!}\,|y|^k.\]
It is easy to see that  for every fixed $n$, 
$\limsup_{k\to \infty}
{\sqrt[2k]{\mu_{2n+2k}}}/{\sqrt[k]{k!}}=2\rho$; thus, for every $\varepsilon >0$ the sum on the right hand side of the above equation converges uniformly for all $|y|\leq \frac{1}{2\rho}-\varepsilon $.
In particular, we can differentiate under the integral sign in \eqref{ift0} any number of times, which implies \eqref{fourier-int-der}.
\end{proof}

\begin{definition}
For $\varphi\in \LT$ we call $f_{\varphi}(z)$ defined by \eqref{ift0} the inverse Fourier   transform of $\varphi(z)$ with respect to $\mom$, denoted by
$\FT^{-1}[\varphi](z)$.
\end{definition}

\begin{lemma} Assume that  for all $n$,  $\varphi_n(\omega)\in \LT$ and that $\varphi(\omega),\psi(\omega)\in \LT$;  if $\varphi_n(\omega)\to  \varphi(\omega)$ in $\LT$, then 
\begin{align}\label{cont1}\lim_{n\to \infty}\int_{-\infty}^\infty \varphi_n(\omega)\psi(\omega)\da=\int_{-\infty}^\infty \varphi(\omega)\psi(\omega)\da.
\end{align}
In particular, if in $\LT$ we have $\varphi(\omega)=\sum_{k=1}^\infty \eta_k(\omega)$, then 
\begin{equation}\label{suma}
\int_{-\infty}^{\infty}\psi(\omega) \sum_{k=0}^\infty\eta_k(\omega)\da=\sum_{k=0}^\infty\int_{-\infty}^\infty \psi(\omega)\eta_k(\omega)\da.
\end{equation}
\end{lemma}
\begin{proof} Follows from 
\[\left|\int_{-\infty}^\infty(\varphi(\omega)-\varphi_n(\omega))\psi(\omega)\da\right|\leq \int_{-\infty}^\infty|\varphi(\omega)-\varphi_n(\omega)|^2\da  \int_{-\infty}^\infty|\psi(\omega)|^2\da\to  0.\]
\end{proof}

\begin{corollary}
$\FT^{-1}$ is an inverse of $\FT$.
\end{corollary} 

\begin{proof} Assume $f\in\LB$; then $\varphi_f=  \sum_{n=0}^{\infty}
(-{\ii})^{n}\K{n}[f](0)\,p_{n}(\omega)\in \LT$. Let $g(t)=\FT^{-1}[\varphi_f(\omega)]$; then 
\eqref{fourier-int-der}, \eqref{suma} and the fact that $p_n(\omega)$ are orthonormal imply
\begin{eqnarray}
\K{n}[g](0)&=&\int_{-\infty}^{\infty}
{\ii}^{n}\,p_{n}(\omega)\,
\sum_{k=0}^\infty (-{\ii})^{k}\K{k}[f](0) p_{k}(\omega)\da\\
&=&\sum_{k=0}^\infty \K{k}[f](0){\ii}^{n}(-{\ii})^{k}\int_{-\infty}^{\infty}\,p_{n}(\omega)p_{k}(\omega)\da
=\K{n}[f](0)
\end{eqnarray}
Thus, also $g^{(n)}(0)=f^{(n)}(0)$ and consequently $g(\omega)=f(\omega)$ over $\st{2}$.
\end{proof}

\begin{proposition}\label{some}
Let $\varphi(\omega)\in\LT$ and $f_{\varphi}(z)$ be given by \eqref{ift0}; if for some fixed $u\in\st{2}$ the
function $\varphi(\omega){\e}^{{\ii}\omega u}$ also belongs to
$\LT$, then for such  $u$ we obtain that
\begin{equation}\label{ft-expand}
\varphi(\omega)e^{{\ii}\omega u }=\sum_{n=0}^{\infty}(-{\ii})^{n}
\K{n}[f_{\varphi}](u)\,p_{n}(\omega),
\end{equation}
where the equality holds in the sense of $\LT$, and
\begin{equation}\label{inde}
\sum_{n=0}^{\infty} |\K{n}[f_{\varphi}](u)|^2=
\noi{\varphi(\omega){\e}^{{\ii}\omega u}}^{2} <\infty.
\end{equation}
\end{proposition}
\begin{proof}
By Proposition \ref{f-phi}, if $u\in\st{2}$,
then equation \eqref{fourier-int-der} holds
for the corresponding $f_{\varphi}$ given by \eqref{ift0}. If
also $\varphi(\omega){\e}^{{\ii}\omega u}\in\LT$, then
\eqref{fourier-int-der} implies that
\begin{equation}\label{proj}
\doti{\varphi(\omega){\e}^{{\ii}\omega u}}{p_{n}(\omega)}=
(-{\ii})^n\K{n}[f_{\varphi}](u).
\end{equation}
Since $\PP$ is a complete
orthonormal system in $\LT$, \eqref{proj} implies
\eqref{ft-expand}, and Parseval's Theorem implies \eqref{inde}.
\end{proof}

\begin{corollary}\label{cdnorm}
If $\varphi(\omega)\in\LT$ and $t\in\Rset$, then also $\varphi(\omega)
{\e}^{{\ii}\omega t}\in\LT$ and  \eqref{inde} implies that for all $t\in \Rset$,
\begin{equation}\sum_{n=0}^{\infty} |\K{n}[f_{\varphi}](t)|^2
=\noi{\varphi(\omega){\e}^{{\ii}\omega t
}}^{2}=\noi{\varphi(\omega)}^{2}.
\end{equation}
\end{corollary}

By Proposition \ref{f-phi} if $\varphi(\omega)\in\LT$ then $f_{\varphi}(z)=\FT^{-1}[\varphi](z)$ given by \eqref{ift0} is analytic on $\st{2}$. Thus, Corollary \ref{cdnorm} with $t=0$ implies the following corollary.
\begin{corollary}
If $\varphi(\omega)\in\LT$ then $f_{\varphi}(z)=\FT^{-1}[\varphi](z)$  belongs to $\LB$.
\end{corollary}
Note that we have already shown that if $f(z)\in\LB$ then $\varphi_f(\omega)=\mathcal{F}[f(z)](\omega)\in\LT$; thus, $\mathcal{F}$ and $\mathcal{F}^{-1}$ are bijections between $\LB$ and $\LT$. In particular, $\noi{\mathcal{F}[\mm^{(k)}(z)](\omega)}^2=\int_{-\infty}^\infty\omega^{2k}\da=\mu_{2k}<\infty$; similarly, 
$\noi{\mathcal{F}[\K{k}[\mm(z)]](\omega)}^2=\int_{-\infty}^\infty p_k(\omega)^{2}\da=1$.
Thus, $\mm^{(k)}(z)$ and $\K{k}[\mm](z)$  belong to $\LB$ for all $k$. 

\begin{lemma}\label{euw}
For every $\varepsilon$, $0< \varepsilon< \frac{1}{2\rho}$, and every polynomial $P(\omega)$, if $u\in \mathds{S}\!\left(\frac{1}{2\rho}-\varepsilon \right)$, then $\noi{P(\omega){\e}^{\ii u\; \omega}}\leq \noi{P(\omega){\e}^{\left(\frac{1}{2\rho}-\varepsilon\right) |\omega|}}<\infty$.
\end{lemma}

\begin{corollary}\label{sum-squares}
Let $0<\varepsilon<\frac{1}{2\rho}$ and let $m$ be any non-negative integer; then for all $u\in \mathds{S}(\frac{1}{2\rho}-\varepsilon)$
\begin{equation}\label{sumis2}
\sum_{n=0}^{\infty}|(\K{n}\circ\K{m})[\mm](u)|^2
< \noi{p_m(\omega){\e}^{(\frac{1}{2\rho}-\varepsilon)|\omega| }}^2<\infty.
\end{equation}
In particular, for all real $t$ we have
\begin{equation}\label{sumis1}
\sum_{n=0}^{\infty} |(\K{n}\circ \K{m})[\mm](t)|^2=1.
\end{equation}
\end{corollary}

\begin{proof}  Lemma \ref{euw} implies that ${\ii}^mp_m(\omega){\e}^{\ii u}$ belongs to $\LT$; thus, we can apply Proposition \ref{some} with $\varphi(\omega)={\ii}^mp_m(\omega)$, in which case
$f_{\varphi}(z)=\K{m}[\mm](z)$ and
\begin{equation}\label{sumsqpar}
\sum_{n=0}^{\infty} |(\K{n}\circ\K{m})[\mm](u)|^2=
\noi{p_m(\omega){\e}^{{\ii}\omega u}}^{2}.
\end{equation}
If $u$ is real we get \eqref{sumis1}; if  $u\in \mathds{S}(\frac{1}{2\rho}-\varepsilon)$ then the claim follows from Lemma \ref{euw}.
\end{proof}

\subsection{Space $\LL$}

\begin{corollary}\label{cdscpr} Assume that $\varphi(\omega),\psi(\omega)\in\LT$ and that $u\in\Rset$; then  \eqref{proj} implies 
\begin{align*}
\doti{\varphi(\omega)}{\psi(\omega)} = \doti{\varphi(\omega){\e}^{{\ii}\omega u}}{\psi(\omega){\e}^{{\ii}\omega u}}&=
\sum_{n=0}^{\infty} (-\ii)^n \K{n}[f_\varphi](u)\; \overline{({-\ii})^n\K{n}[f_\psi](u)}\nonumber\\
&= \sum_{n=0}^{\infty}  \K{n}[f_\varphi](u)\; \overline{\K{n}[f_\psi](u)}.
\end{align*}
\end{corollary}

\begin{corollary}\label{cdconv} Assume that $\varphi(\omega),\psi(\omega)\in\LT$ and that $u, t\in \Rset$; then 
\begin{align*}
\sum_{n=0}^{\infty}(-1)^n \K{n}[f_{\varphi}](u)\, \K{n}[f_{\psi}](t-u)=\int_{-\infty}^{\infty}\varphi(\omega)\,\psi(\omega)\,{\e}^{{\ii}\omega t}\da.
\end{align*}
\end{corollary}
\begin{proof}
Since $t\in \Rset$, we have
\begin{align*}
\left|\int_{-\infty}^{\infty}\varphi(\omega)\,\psi(\omega)\,{\e}^{{\ii}\omega t}\da\right|\leq &\int_{-\infty}^{\infty}|\varphi(\omega)|^2\da\;\int_{-\infty}^{\infty}|\psi(\omega)|^2\da<\infty.
\end{align*}
Using \eqref{ft-expand} and \eqref{cont1} and since $\varphi(\omega),\psi(\omega)\in \LT$, we obtain that for all $u,t\in\Rset$, 
\begin{align*}
\int_{-\infty}^{\infty}\varphi(\omega)\,\psi(\omega)\,{\e}^{{\ii}\omega t}\da&=\int_{-\infty}^{\infty}\varphi(\omega)\,{\e}^{{\ii}\omega u}\psi(\omega)\,{\e}^{{\ii}\omega (t-u)}\da\\
&=\int_{-\infty}^{\infty} \sum_{n=0}^{\infty}(-{\ii})^{n}
\K{n}[f_{\varphi}](u)p_{n}(\omega)\psi(\omega)\,{\e}^{{\ii}\omega (t-u)}\da\\
&=\sum_{n=0}^{\infty}(-1)^{n}
\K{n}[f_{\varphi}](u)\,\int_{-\infty}^{\infty} {\ii}^{n}p_{n}(\omega)\psi(\omega)\,{\e}^{{\ii}\omega (t-u)}\da\\
&=\sum_{n=0}^{\infty}(-1)^{n} \K{n}[f_{\varphi}](u)\, \K{n}[f_{\psi}](t-u)
\end{align*}
\end{proof}

\begin{definition}
$\LL$ is the vector space of the restrictions of functions from $\LB$ to real arguments.
\end{definition}

The following Proposition follows directly from Corollaries  \ref{cdnorm} -- \ref{cdconv}.
\begin{proposition}[\!\!\cite{IG5}]\label{local-space-gen}
We can introduce a locally defined scalar product, an associated
norm and a convolution of functions in  \LL\  by the following
sums which are independent of a real $u$:
\begin{align}
\norm{f}^2&=\sum_{n=0}^{\infty}|\K{n}[f](u)|^2=
\noi{\FT[f](\omega)}^2;\label{nor}\\
\Mscal{f}{g}&=\sum_{n=0}^{\infty}\K{n}[f](u)\overline{\K{n}[g](u)}
=\doti{\FT[f](\omega)}{\FT[g](\omega)};\label{scl}\\
(f \ast_{\!\Mi} g)(t)&=\sum_{n=0}^{\infty}(-1)^n\K{n}[f](u)
\K{n}[g](t-u)\label{convolution}
=\int_{-\infty}^{\infty}
{\FT[f](\omega)}\;{\FT[g](\omega)}{\e}^{{\ii}\omega t}
\da.
\end{align}
\end{proposition}
 \begin{corollary}
 The space $\LL$ with the scalar product  $\Mscal{f}{g}$ is isomorphic to the space $\LT$  with the scalar product $\doti{\varphi_f}{\varphi_g}$ via $f(t)\mapsto \varphi_f(\omega)=\FT[f](\omega)$ given by \eqref{ft0}, with the inverse given by $\varphi(\omega)\mapsto f_\varphi(t)=\FT^{-1}[\varphi](t)$ given by \eqref{ift0}.
\end{corollary}
Letting $u=0$, $u=t$ and $u=t/2$ in in \eqref{convolution}
we get the following lemma.
\begin{lemma}[\!\cite{IG5}]\label{con}
For every $f,g\in \LL$ and for every $t\in\Rset,$
\begin{equation}\label{symm}
\sum_{k=0}^{\infty}
(-1)^k\,\K{k}[f](t)\,\K{k}[g](0)
=\sum_{k=0}^{\infty}(-1)^k\,\K{k}[f](0)\, \K{k}[g](t)
=\sum_{k=0}^{\infty}(-1)^k\,\K{k}[f](t/2)\, \K{k}[g](t/2).
\end{equation}
\end{lemma}

\subsection{Relationship of $\FT[f](\omega)$ with the usual Fourier transform}
If a moment distribution function $\alpha(\omega)$ is absolutely continuous then
$\da= \wght(\omega)\mathrm{d}\omega$ for some
non-negative weight function $\wght(\omega)$; in such a case for every  $f(t)\in \LL$ equation 
\eqref{ift0} implies
\begin{equation}\label{oldft}
f(t)=\int_{-\infty}^{\infty}\FT[f](\omega)\;{\e}^{\ii\omega t}\wght(\omega)\;{\rm d}\omega.
\end{equation}

\begin{proposition}[\!\!\cite{IG5}]\label{weight-space}
Assume that for a function $f(t)$ there exists a function $\widehat{f}(\omega)$ such that
$f(t)=\frac{1}{2\pi}\int_{-\infty}^{\infty}
\widehat{f}(\omega)\,{\e}^{\ii\omega t}\,{\rm d}\omega$ for all
$t\in\Rset$; then $f(t)\in\LL$ if and only if
$\int_{-\infty}^{\infty}|\widehat{f}(\omega)|^2
\wght(\omega)^{-1}\,{\rm d}\omega<\infty,$ in which case
$\widehat{f}(\omega)=2\pi\, \FT[f](\omega)\wght(\omega)$.
\end{proposition}

\begin{proof}If $f(t)\in\LL$ then \eqref{oldft} holds and this implies, by the uniqueness of the Fourier transform, that $\widehat{f}(\omega)=2\pi\; \FT[f](\omega)\wght(\omega)$ and consequently
\begin{align*}
\int_{-\infty}^\infty \frac{|\widehat{f}(\omega)|^2}{\wght(\omega)}{\rm d}\omega=\int_{-\infty}^\infty \frac{{4\pi^2}| \FT[f](\omega)|^2 \wght(\omega)^2}{\wght(\omega)}{\rm d}\omega ={4\pi^2}\noi{\FT[f](\omega)}^2<\infty.
\end{align*}
In the opposite direction, 
\begin{align*}
\int_{-\infty}^\infty \frac{|\widehat{f}(\omega)|^2}{\wght(\omega)^2}\wght(\omega){\rm d}\omega=
\int_{-\infty}^\infty \frac{|\widehat{f}(\omega)|^2}{\wght(\omega)}{\rm d}\omega<\infty
\end{align*}
implies ${\widehat{f}(\omega)}/{\wght(\omega)}\in \LT$ and thus 
\[f(t)=\frac{1}{2\pi}\int_{-\infty}^{\infty}
\widehat{f}(\omega)\;{\e}^{\ii\omega t}\;{\rm d}\omega=\int_{-\infty}^{\infty}\frac{1}{2\pi}
\frac{\widehat{f}(\omega)}{\wght(\omega)}\;{\e}^{\ii\omega t}\;\wght(\omega){\rm d}\omega\in \LL
\]
and $\widehat{f}(\omega)/({2\pi}\wght(\omega))= \FT[f](\omega)$.
\end{proof}
\section{Uniform convergence of chromatic expansions}

\begin{theorem}\label{unif-con}
Assume $f\in\LB$; then for all $u\in\Rset$ and $\frac{1}{2\rho}>\varepsilon>0$,
the chromatic series $\CE[f,u](z)$ of $f(z)$ converges to
$f(z)$ uniformly on the strip
$\mathds{S}(\frac{1}{2\rho}-\varepsilon)$.
\end{theorem}
\begin{proof}
Assume $u\in\Rset$ and let $\phi(\omega)\equiv 1$; then $f_\varphi(z-u)=\mm(z-u)$. By applying \eqref{fourier-int-der} 
we obtain that for all $z\in\st{2}$,
\begin{equation}\label{exp-CA}\CA[f,n,u](z)=
\int_{-\infty}^{\infty}\sum_{k=0}^{n}(-{\ii})^{k}
\K{k}[f](u)\,
p_{k}(\omega){\e}^{\ii \omega (z-u)}\da.
\end{equation}

Since $f\in\LB$, we obtain 
$\FT[f](\omega)\in\LT$ and since $u\in\Rset$ equation  
\eqref{ft-expand} from Proposition ~\ref{some} implies that in $\LT$
\[
\FT[f](\omega)\;{\e}^{\ii\omega u}=
\sum_{k=0}^{\infty}(-{\ii})^{k}
\K{k}[f](u)\,p_{k}(\omega).
\]
Thus,
\begin{align}\label{f-ex}
f(z)=\int_{-\infty}^{\infty}\FT[f](\omega){\e}^{\ii \omega u}{\e}^{\ii \omega (z-u)}\da= \int_{-\infty}^{\infty}\sum_{k=0}^{\infty}(-{\ii})^{k}
\K{k}[f](u)\,p_{k}(\omega){\e}^{\ii \omega (z-u)}\da.
\end{align}
Consequently, from \eqref{exp-CA} and \eqref{f-ex},
\begin{align*}
 |f(z)-\CA[f,n,u](z)|
&\leq\int_{-\infty}^{\infty}\left|
\sum_{k=n+1}^{\infty}(-{\ii})^{k}
\K{k}[f](u)\,p_{k}(\omega){\e}^{\ii \omega (z-u)}
\right|\da\nonumber\\
&\leq\left(\!\int_{-\infty}^{\infty}\!
\left|\sum_{k=n+1}^{\infty}(-{\ii})^{n}
\K{n}[f](u)\,p_{n}(\omega)\right|^2\!\da
\int_{-\infty}^{\infty}|{\e}^{\ii\omega(z-u)}|^2\da\!\right)^{1/2}
\end{align*}
This implies that for $z\in \mathds{S}(\frac{1}{2\rho}-\varepsilon)$ and $u\in \Rset$,  using Corollary \ref{cdnorm},  we obtain
\begin{equation}\label{convergence}
|f(z)-\CA[f,n,u](z)|\leq\!\left(\sum_{k=n+1}^{\infty}|
\K{n}[f](u)|^2
\!\!\int_{-\infty}^{\infty}
{\e}^{(\frac{1}{\rho}-2\varepsilon)|\omega|}\da\right)^{1/2}.
\end{equation}
Consequently, \eqref{sumis1} of Corollary \ref{sum-squares} with $m=0$ and Lemma \ref{meuw} 
imply that $\CE[f,u](z)$ converges to $f(z)$ uniformly on
$\mathds{S}(\frac{1}{2\rho}-\varepsilon)$.
\end{proof}

\begin{proposition}\label{rep} Space $\LB$ consists precisely
of functions
\begin{equation}\label{ser}f(z)=\sum_{k=0}^{\infty}a_{k} \K{k}[\mm](z)
\end{equation}
restricted to $\mathds{S}(\frac{1}{2\rho})$, 
where $a=\Langle
a_k\Rangle_{k\in\Nset}$ is a complex sequence in $l^2$.
\end{proposition}

\begin{proof}
Assume $a\in l^2$; by Corollary \ref{sum-squares} with $m=1$, for every
$\varepsilon>0$, if $z\in \mathds{S}(\frac{1}{2\rho}-\varepsilon)$ then
\begin{align*}
\sum_{k=N+1}^{\infty}|a_{k}\K{k}[\mm](z)|&\leq
\left(\sum_{k=N+1}^{\infty}|a_{k}|^2\sum_{k=N+1}^{\infty}
|\K{k}[\mm](z)|^2\right)^{1/2}\\
&\leq \left(\sum_{k=N+1}^{\infty}|a_{k}|^2\right)^{1/2}
\noi{{\e}^{(\frac{1}{2\rho}-\varepsilon) |\omega| }},
\end{align*}
which implies that the series in \eqref{ser} converges absolutely and uniformly on
$\mathds{S}(\frac{1}{2\rho}-\varepsilon)$. Consequently,
$f(z)$ is analytic on $\st{2}$.  We now show that  
\begin{equation}\label{circ}
\K{m}[f](z)=\sum_{k=0}^\infty a_k(\K{k}\circ\K{m})[m](z)
\end{equation}
with the series converging absolutely and uniformly on $\mathds{S}(\frac{1}{2\rho}-\varepsilon)$.
For every $m$ equation \eqref{sumis2} in Corollary  \ref{sum-squares} implies 
\begin{align*}
\sum_{k=N+1}^\infty|a_k(\K{k}\circ\K{m})[\mm](z)|&\leq 
\left(\sum_{k=N+1}^\infty|a_k|^2\sum_{k=N+1}^\infty|(\K{m}\circ\K{k})[\mm](z)|^2\right)^{1/2}\\
&\leq \left(\sum_{k=N+1}^\infty|a_k|^2\right)^{1/2}\noi{p_m(\omega){\e}^{(\frac{1}{2\rho}-\varepsilon)|\omega| }}
\end{align*}
Thus, for every $m$ the series $\sum_{k=0}^\infty a_k(\K{k}\circ\K{m})[\mm](z)$ also converges absolutely and uniformly on  $\mathds{S}(\frac{1}{2\rho}-\varepsilon)$.
This fact and the inequality 
\begin{align*}
|(D\circ\K{k}\circ\K{n})[\mm](z)|\leq \gamma_n|(\K{k}\circ\K{n+1})[\mm](z)|+|\beta_n(\K{k}\circ\K{n})[\mm](z)|+
\gamma_{n-1}|(\K{k}\circ\K{n-1})[\mm](z)|
\end{align*}
which follows from \eqref{three-term} imply that the series at the righthand side of the equation 
\begin{align*}
(D\circ\K{n})[f](t)=\sum_{k=0}^\infty a_k(D\circ\K{k}\circ\K{n})[\mm](t-u)
\end{align*}
also converges absolutely and uniformly, thus justifying that equation. Equation \eqref{circ} now follows by a straightforward induction on $n$. From this equation and \eqref{orthonorm} we conclude that 
\begin{align*}
\K{m}[f](0)=
\sum_{n=0}^{\infty}a_n(\K{m}\circ\K{n})[\mm](0)=(-1)^ma_m.
\end{align*}
Consequently, $\sum_{n=0}^{\infty} |\K{n}[f](0)|^2=\sum_{n=0}^{\infty}
|a_n|^2<\infty$ and so $f(z)\in \LB$.  Proposition
\ref{unif-con} provides the opposite direction.
\end{proof}

Note that the proof of Proposition \ref{rep} also proves the following corollary.

\begin{corollary}\label{postdiff} For $f(t)\in\LB$ and $u\in \Rset$ we have
\begin{enumerate}
\item 
\begin{align}\label{insidediff}
\K{n}[f](z)=\sum_{k=0}^\infty(-1)^k\K{k}[f](u)(\K{k}\circ\K{n})[\mm](z-u)
\end{align}
with the series converging absolutely and uniformly on $\Rset$.

\item For every $0<\varepsilon<1/(2\rho)$ functions $f(z)\in\LB$ are bounded on the strip
$\mathds{S}(\frac{1}{2\rho}-\varepsilon)$  
\begin{equation}\label{stripbound}
|f(z)|\leq \left(\sum_{n=0}^{\infty}|\K{n}[f](0)|^2\right)^{1/2}
\noi{{\e}^{(\frac{1}{2\rho}-\varepsilon) |\omega|}}.
\end{equation}
\item 
If $f(z)$ is analytic in a neighbourhood of zero and satisfies $\sum_{n=0}^\infty |\K{n}[f](0)|^2<\infty$, then its chromatic expansion provides the analytic continuation of such $f(z)$ on $\st{2}$.
\end{enumerate}
\end{corollary}

\subsection{Chromatic expansions of functions in \LL}
Note that for $f\in\LL$ we have $\CE[f,u](t)=(f\ast_{\Mi}\mm)(t)$; thus, $(f\ast_{\Mi}\mm)(t)=f(t)$.
Also, 
note that using \eqref{sumis1} with $m=0$ we get that for $t,u\in \Rset$
\begin{align}
|f(t)-\CA[f,n,u](t)|&\leq\sum_{k=n+1}^{\infty}
|\K{k}[f](u)\K{k}[\mm](t-u)|\nonumber\\
&\leq\left(\sum_{k=n+1}^{\infty}|\K{k}[f](u)|^2\right)^{1/2}
\left(1-\sum_{k=0}^{n}|\K{k}[\mm](t-u)|^2\right)^{1/2}.\label{error}
\end{align}

Let
\begin{align*}
E_n(t)=\left(1-\sum_{k=0}^{n}|\K{k}[\mm](t)|^2\right)^{1/2};
\end{align*}
then, using Lemma \ref{CD}, we obtain
\begin{align*}
E_n^\prime(t)=\gamma_n\;\mathfrak{Re}(\K{n+1}[\mm](t)\;\overline{\K{n}[\mm](t)})
\left(1-\sum_{k=0}^{n}|\K{k}[\mm](t)|^2\right)^{-1/2}.
\end{align*}
Since $(D^{k}\circ\K{n})[\mm](0)=0$ for all $0\leq k\leq n-1$, the smallest powers in the Taylor expansions of $\K{n+1}[\mm](t)$ and of $\K{n}[\mm](t)$ at $t=0$ are $n+1$ and $n$, respectively. Thus, $E_n^{(k)}(0)=0$ for all $k\leq 2n+1$. Consequently,
$E_{n}(0)=0$ and $E_{n}(t)$ is very flat around $t=0$. For that reason chromatic expansions provide
excellent local approximations of functions in $\LL$.\\

Let $u$ be a fixed real
parameter and $t$ a variable ranging over $\Rset$; consider functions\\ $B^{u}_{n}(t)=
(-1)^n\K{n}[\mm](t-u)$. Since the value of $\Mscal{B^{u}_{n}(t)}{B^{u}_{m}(t)}$ does not depend on $t$, using \eqref{orthonorm},
\begin{align}
\Mscal{B^{u}_{n}(t)}{B^{u}_{m}(t)}
&=\sum_{k=0}^{\infty}(-1)^{m+n}(\K{k}\circ\K{n})[\mm](t-u)
\overline{(\K{k}\circ\K{m})[\mm](t-u)}\Big|_{t=u}=\delta(m-n).
\end{align}
Thus, the family $\{B^{n}_{u}(t)\}_{n\in\Nset}$ is orthonormal in
$\LL$.  

\begin{proposition}[\!\!\cite{IG5}]\label{in-LL}
The chromatic expansion $\CE[f,u](t)$ of $f(t)\in\LL$ is the (generalised) 
Fourier series of $f(t)$ with respect to the orthonormal basis
$\{B^u_n(t)\}_{n\in\Nset}$.  The chromatic
expansion converges to $f(t)$
in \LL; thus, $\{B_n^u(t)\}_{n\in\Nset}$ is a complete
orthonormal basis of \LL.
\end{proposition}
\begin{proof}
For $f\in\LL$ the scalar product $\Mscal{f(t)}{(-1)^n\K{n}[\mm](t-u)}$ does not depend on $t\in\Rset$; thus, by \eqref{orthonorm},
\begin{align}\label{projf}
\Mscal{f(t)}{(-1)^n\K{n}[\mm](t-u)}&=
\sum_{k=0}^{\infty}(-1)^n
\K{k}[f](t)\overline{(\K{k}\circ\K{n})[\mm](t-u)}\Big|_{t=u}\\
&= \sum_{k=0}^{\infty}(-1)^n
\K{k}[f](u)\overline{(\K{k}\circ\K{n})[\mm](0)}\nonumber\\
&=\K{n}[f](u).\nonumber
\end{align}

Since  $\K{k}_t[f(t)-\CA[f,n,u](t)]|_{t=u}$ equals $0$ for
$k\leq n$ and equals $\K{k}[f](u)$ for $k>n$, we obtain
$\norm{f-\CA[f,n,u]}=
\sum_{k=n+1}^{\infty}\K{k}[f](u)^2\to  0$.
\end{proof}

\subsection{Chromatic expansions and linear operators on \LL}
Assume that $g(t),g^\prime(t)\in\LL$ and that for every fixed $h$ also $f(t,h)\in\LL$; we define $\mlim f(t,h)$ to be the limit in the sense of $\LL$, i.e.,  $\mlim f(t,h)=g(t)$ if an only if  $\lim_{h\to  0}\norm{f(t,h)-g(t)}=0$. 
\begin{lemma}\label{deriv}
Assume that $f(t),f^{\prime}(t),f^{\prime\prime}(t)\in \LL$; then
\begin{align*}
\mlim \left(\frac{f(t+h)-f(t)}{h}\right)=f^{\prime}(t).
\end{align*}
\end{lemma}
\begin{proof} Since
\begin{align*}
\frac{f(t+h)-f(t)}{h}-f^\prime(t)=\int_{-\infty}^\infty \FT[f](\omega)\left(\frac{{\e}^{\ii \omega h}-1}{h}-\ii \omega \right) 
{\e}^{\ii \omega t}\da,
\end{align*}
we obtain
\begin{align*}
\norm{\frac{f(t+h)-f(t)}{h}-f^\prime(t)}^2&=\noi{  \FT[f](\omega)\left(\frac{{\e}^{\ii \omega h}-1}{h}-\ii \omega \right) }^2\\
&=\int_{-\infty}^{\infty}| \FT[f](\omega)|^2\left|\frac{{\e}^{\ii w h}-1}{h}-\ii \omega \right|^2\da.
\end{align*}
One can verify that 
\begin{align*}
\left|\frac{{\e}^{\ii \omega h}-1}{h}-\ii \omega \right|^2=\omega^2\left(\left(\frac{\sin\frac{h\omega}{2}}{\frac{h \omega}{2}}\right)^{\!2}-2\,\frac{\sin h\omega}{h \omega}+1\right)\leq\frac{\omega^4h^2}{4};
\end{align*}
thus,
\begin{align*}
\norm{\frac{f(t+h)-f(t)}{h}-f^\prime(t)}^2\leq\frac{h^2}{4}\int_{-\infty}^{\infty}|\omega^{2} \FT[f](\omega)|^2\da=\frac{h^2}{4}\norm{ \FT[f^{\prime\prime}](\omega)}^2,
\end{align*}
which converges to zero as $h\to  0$.
\end{proof}

\begin{lemma}
Let  $A$ be a linear operator on $\LL$ which is continuous with respect to the norm of $\LL$ and which is shift invariant, i.e., such that for every fixed $h$, $A[f(t+h)]=A[f](t+h)$ for all $f\in \LL$. Assume also that $f(t), f^\prime(t), f^{\prime\prime}(t)\in \LL$, then $A[f^{\prime}(t)]=(A[f(t)])^{\prime}$.
\end{lemma}
\begin{proof} Using Lemma \ref{deriv} and the facts that $A$ is continuous with respect to the norm of \LL\ and that is  linear and shift invariant, we obtain 
\begin{align*}A[ f^\prime(t)]&=A\left[\mlim\frac{ f(t+h)- f(t)}{h}\right]=
\mlim A\left[\frac{ f(t+h)- f(t)}{h}\right]\\
&=\mlim\frac{A[ f](t+h)-A[ f](t)}{h}=(A[ f(t)])^\prime.
\end{align*}
\end{proof}

Since $\mm^{(k)}(t)\in \LL$ for all integers $k>0$, by induction it follows that $A[\mm^{(k)}(t)]=(A[\mm(t)])^{(k)}$ for all integers $k$, which in turn implies that also  $A[\K{k}[\mm(t)]]=\K{k}[A[\mm(t)]]$.  Using Proposition  \ref{in-LL}, we obtain the following corollary.
\begin{corollary} For every linear operator on $\LL$ which is continuous with respect to the norm of $\LL$ and which is shift invariant and for all $f(t)\in \LL$ and $u\in\Rset$,
\begin{equation*}
A[f](t)=\sum_{n=0}^{\infty} \,(-1)^{n}\K{n} [f](u)\, \K{n}[A[\mm]](t-u)
=(f\ast_{\Mi}A[\mm])(t).
\end{equation*}
\end{corollary}
Consequently, the action of such $A$ on any function in $\LL$
is uniquely determined by $A[\mm(t)]$. Thus, $\mm(t)$ plays the same role which $\sinc(t)$ plays  in the standard signal processing
paradigm based on the Shannon expansion.

\begin{proposition}
If  $A$ is a continuous linear operator on \LL\ such that 
\begin{equation}\label{dcomm}
(A\circ D^n)[\mm](t)=(D^n\circ A)[\mm](t)\end{equation}
for all $n$, then such $A$ must be shift invariant.
\end{proposition}
\begin{proof}
The continuity of $A$ on \LL, equation \eqref{dcomm} and Lemma~\ref{con} imply that for every $f\in \LL$,
\begin{align*}
A[f](t)&= \sum_{n=0}^{\infty} \,(-1)^{n}\K{n} [f](0)\,
\K{n}[A[\mm]](t) =\sum_{n=0}^{\infty} \,(-1)^{n} \K{n}[A[\mm]](0)\, \K{n} [f](t).
\end{align*}
Since operators $\K{n}$ are shift invariant, such $A$ must also be 
shift invariant.
\end{proof}

\section{Examples}\label{examples}
We now present a few examples of chromatic derivatives and
chromatic expansions, associated with several classical
families of orthogonal polynomials. The formulas below were derived with the help of  the \emph{Mathematica}\texttrademark\  software.

\subsection{Example 1: the Legendre polynomials}
Let $L_n(\omega)$ be the Legendre polynomials; if we set
$p_{n}^{\scriptscriptstyle{L}}(\omega)=
\sqrt{2n+1}\,L_n(\omega/\pi)$ then
\begin{equation*}
\int_{-\pi}^{\pi}
p_{n}^{\scriptscriptstyle{L}}(\omega)
p_{m}^{\scriptscriptstyle{L}}(\omega)\;\frac{1}{2\pi}\,d \omega=\delta(m-n).
\end{equation*}
The corresponding recursion coefficients in equation
\eqref{poly} are given by the formula 
\[\gamma_n=\frac{\pi
(n+1)}{\sqrt{4(n+1)^2-1}};\]
the corresponding space $\LT$ is
$L^2[-\pi,\pi]$. In this case
\[\mm(z)=j_0(\pi z)=\sinc(t)\]
and
\[\K{n}[\mm](z)=(-1)^n\sqrt{2n+1}j_n(\pi z)\] 
where $j_n(z)$ is the spherical Bessel function of order $n$.\footnote{$j_n(x)=\frac{\pi}{2x}J_{n+1/2}(x)$ where $J_a(x)$ is the Bessel function of the first kind.}\\
In this case $\mu_{2n}=\frac{\pi^{2n}}{2n+1}$ and so $\rho=\lim_{n\to\infty}\left(\frac{\mu_{2n}}{(2n)!}\right)^{\frac{1}{2n}}=0$.
Consequently, $1/\rho=\infty$ and $\mm(z)$ is an entire function.
The space $\LB$ for this particular example
consists of all entire functions whose restrictions to $\Rset$
belong to $L^2$ and which have a Fourier transform supported in
$[-\pi,\pi]$. Proposition~\ref{local-space-gen} implies that in
this case our locally defined scalar product $\Mscal{f}{g}$,
norm $\norm{f}$ and convolution $(f \ast_{\Mi} g)(t)$ coincide
with the usual scalar product, norm and convolution on $L_2$.

\subsection{Example 2: the Chebyshev polynomials of the first kind} Let
$p_{n}^{\scriptscriptstyle{T}}(\omega)$ be the family of
orthonormal polynomials obtained by normalising and rescaling
the Chebyshev polynomials of the first kind, $T_n(\omega)$, by
setting $p_{0}^{\scriptscriptstyle{T}}(\omega)= 1$ and
$p_{n}^{\scriptscriptstyle{T}}(\omega)=\sqrt{2}\;T_n(\omega/\pi)$
for $n>0$.  In this case
\begin{equation*}
\int_{-\pi}^{\pi}p_{n}^{\scriptscriptstyle{T}}(\omega)
p_{m}^{\scriptscriptstyle{T}}(\omega)
\frac{d\omega}{
\pi^2\sqrt{\pi^2-\omega^2}}=\delta(n-m)
\end{equation*}
and the corresponding space $\LB$ contains functions which do not
belong to $L^2$; the corresponding function \eqref{mmz} is
$\mm(z)={\mathrm J}_0(\pi z)\not\in L^2$ and  for $n>0$,
$\K{n}[\mm](z)=(-1)^{n}\sqrt{2}\,{\mathrm J}_n(\pi z)$, where
${\mathrm J}_n(z)$ is the Bessel function of the first kind of
order $n$. In the recurrence relation \eqref{three-term} the
coefficients are given by $\gamma_0=\pi/\sqrt{2}$ and
$\gamma_n=\pi/2$ for $n>0$.

The chromatic expansion of a function $f(z)$ is the Neumann
series of $f(z)$ (see \cite{WAT}),
\begin{equation*}
f(z)=f(u){\mathrm J}_0 (\pi(
z-u))+\sqrt{2}\;\sum_{n=1}^{\infty}\K{n}[f](u){\mathrm J}_n(\pi (z-u)).
\end{equation*}
Thus, the chromatic expansions corresponding to various
families of orthogonal polynomials can be seen as
generalisations of the Neumann series, while the families of
corresponding functions $\{\K{n}[\mm](z)\}_{n\in\Nset}$ can be
seen as generalisations and a uniform representation of some
familiar families of special functions. In this case $\mu_{2n}=\frac{\pi^{2n-1/2}\Gamma(n+1/2)}{n!}$ and again $1/\rho=\infty$.

\subsection{Example 3: the Chebyshev polynomials of the second kind} 
Let $p_{n}^{\scriptscriptstyle{U}}(\omega)$ be the family of
orthonormal polynomials obtained by rescaling
the Chebyshev polynomials of the second kind, $U_n(\omega)$, by
setting $p_{0}^{\scriptscriptstyle{U}}(\omega)= 1$ and
$p_{n}^{\scriptscriptstyle{U}}(\omega)=U_n(\omega/\pi)$
for $n>0$.  In this case
\begin{equation*}
\int_{-\pi}^{\pi}p_{n}^{\scriptscriptstyle{U}}(\omega)
p_{m}^{\scriptscriptstyle{U}}(\omega)
\frac{2}{\pi^2}\sqrt{1-\left(\frac{\omega}{\pi}\right)^2}
d\omega=\delta(n-m)
\end{equation*}
The recursion coefficients are given by $\gamma_n=\pi/2$ and $\mm(z)=J_0(\pi z)+J_2(\pi z)$, while\\ $\K{n}[\mm](z)=(-1)^n(J_n(\pi z)+J_{n+2}(\pi z)$ and $\mu_{2n}=\frac{\pi^{2n-1/2}\Gamma(n+1/2)}{2\Gamma(n+2)}$. Again $\rho=0$ and $\mm(z)$ is entire.

\subsection{Example 4: the Gegenbauer polynomials} We rescale and normalise the Gegenbauer polynomials $C^{(a)}_n(\omega)$ by setting 
\[p^{G}_n(\omega)=\frac{\sqrt{\Gamma(a+1/2)\Gamma(n+1)(n+a)}\Gamma(a)}{\pi^{1/4}2^{1/2-a}\sqrt{\Gamma(n+2a)\Gamma(a+1)}}C^{(a)}_n\left(\frac{\omega}{\pi}\right)\]
Then $p^{G}_0(\omega)=1$ and for $a>-1/2$  such polynomials are orthonormal with respect to weight \\$\frac{\Gamma(a+1)}{\pi^{3/2}\Gamma(a+1/2)}\left(1-\left(\frac{\omega}{\pi}\right)^2\right)^{a-1/2}$. The recursion coefficients are given by $\gamma_n=\frac{\pi}{2}\sqrt{\frac{(n+1)(n+2a)}{(n+a)(n+a+1)}}$, and \\$\mm(z)=_0\!\!F_1(;1+a;-\pi^2z^2/4)$, where $_0\!F_1(;x;y)$ is the confluent hypergeometric function; 
\[\K{n}[\mm](z)=2a{\ii}^ne^{-\ii\pi z}\sqrt{\frac{(a+n)\Gamma(2a)}{a\Gamma(n+1)\Gamma(2a+n)}}
\sum_{m=0}^n\binom{m}{n}\frac{(-1)^{m+n}\Gamma(2a+m+n)}{\Gamma(2a+m+1)}  {}_1\!F_1(1/2+a+m,2a+m+1,2\ii\pi z)\]
where ${}_1\!F_1(x,y,z)$ is the Kummer confluent hypergeometric function.

\subsection{Example 5:  the Jacobi polynomials} The Jacobi polynomials generalise several classical families of orthogonal polynomials, such as the Legendre, the Chebyshev and the Gegenbauer polynomials. Thus, for $a,b>-1$ we let 
\[W(a,b)=\int_{-\pi}^{\pi}\left(1-\frac{\omega}{\pi}\right)^a\left(1+\frac{\omega}{\pi}\right)^bd\omega=\pi\left(\frac{_2F_1(1,-a;b+2;-1)}{b+1}+\frac{_2F_1(1,-b;a+2;-1)}{a+1}\right)\]
where $_2F_1$ is the hypergeometric function.
We rescale and normalise the standard Jacobi polynomials $\mathcal{J}_n^{(a,b)}(x)$ to obtain polynomials $p^{J}_n(\omega)$ so that
\[\int_{-\pi}^{\pi}p^J_n(\omega)p^{J}_m(\omega)\frac{\left(1-\frac{\omega}{\pi}\right)^a\left(1+\frac{\omega}{\pi}\right)^b}{W(a,b)}d\omega=\delta(m-n)\]
by setting
\[ p^J_n(\omega)=\left(\frac{(2n+a+b+1)\Gamma(a+1)\Gamma(b+1)\Gamma(n+1)\Gamma(n+a+b+1)}{\Gamma(2+a+b)\Gamma(n+a+1)\Gamma(n+b+1)}\right)^{1/2}\mathcal{J}_n^{(a,b)}\left(\frac{\omega}{\pi}\right)\]
The recursion coefficients for such orthonormal polynomials are given by
\begin{align*}
\gamma_n&=\frac{2\pi}{2n+a+b+2}\sqrt{\frac{(n+1)(n+a+1)(n+b+1)(n+a+b+1)}{(2n+a+b+1)(2n+a+b+3)}}\\
\beta_n&=\frac{\pi\,(a^2-b^2)}{(2n+a+b+2)(2n+a+b)}.
\end{align*}
In this case
\[\mm(z)=\sgn(a+b+1)\sqrt{\frac{\pi\; 2^{a+b+1}\Gamma(a+1)\Gamma(b+1)}{\Gamma(a+b+2)W(a,b)}}\;{\e}^{-\ii \pi t}\; _1F_1(b+1;a+b+2; 2\ii\pi z)\]
where $_1F_1$ is the confluent hypergeometric function, and 
\begin{align*}
\K{n}[\mm](z)=&{\ii}^n{\e}^{-\ii\,\pi\,z}\, \Gamma(a+1)\sqrt{\frac{\pi\,2^{a+b+1}(2n+a+b+1)\Gamma(n+b+1)}{W(a,b)\Gamma(n+a+1)\Gamma(n+1)\Gamma(n+a+b+1)}}\\
&\hspace*{5mm}  \sum_{m=0}^n{{n}\choose{m}}\frac{(-1)^{m}\Gamma(n+m+a+b+1)}{\Gamma(m+a+b+2)}\;_1F_1(m+b+1;m+a+b+2; 2\ii\,\pi\,z)
\end{align*}

\subsection{Example 6:  the Hermite
polynomials}

Let $H_n(\omega)$ be the Hermite polynomials; then polynomials
$p_{n}^{\scriptscriptstyle{H}}(\omega)=
(2^{n}n!)^{-1/2}H_n(\omega)$  satisfy
\begin{equation*}\int_{-\infty}^{\infty}
p_{n}^{\scriptscriptstyle{H}}(\omega)
p_{m}^{\scriptscriptstyle{H}}(\omega)
\;{\e}^{-\omega^2}\;\frac{d\omega}{\sqrt{\pi}}
=\delta(n-m).
\end{equation*}
The corresponding space $\LB$ contains entire functions whose
Fourier transform $\widehat{f}(\omega)$ satisfies\\
$\int_{-\infty}^{\infty} |\widehat{f}(\omega)|^2
\;{\e}^{\omega^2} d \omega<\infty$. In this case the space
$\LL$ contains non-bandlimited signals; the corresponding
function defined by \eqref{mmz} is $\mm(z)={\e}^{-z^2/4}$ and
$\K{n}[\mm](z)= (-1)^{n}(2^{n}\,n!)^{-1/2}\,z^{n}
{\e}^{-z^2/4}$. The corresponding recursion coefficients are
given by $\gamma_n=\sqrt{(n+1)/2}$. In this case $\mu_{2n}=\Gamma(n+1/2)/\sqrt{\pi}$ and thus $1/\rho=\infty$.

\subsection{Example 7: the Laguerre polynomials} The Laguerre polynomials\footnote{Our definition of the Laguerre polynomials is slightly different from the usual one - we multiply the odd degree polynomials by $-1$ to ensure that the leading coefficient is always positive.} $p^{L}_n(\omega)$ satisfy 
\begin{align*}
\int_{0}^{\infty}p^{L}_n(\omega)p^{L}_m(\omega){\e}^{-\omega}d\omega=\delta(m-n)
\end{align*}
The recursion coefficients are given by $\gamma_n=n+1$ and $\beta_n=-(2n+1)$. For this family 
\begin{align}
&\mm(z)=\frac{1}{1-\ii z};&&
\K{n}[\mm](z)=\frac{1}{1-\ii z}\left(\frac{-z}{1-\ii z}\right)^n.&
\end{align}

Since $\mu_n=\int_{0}^{\infty}\omega^n{\e}^{-\omega}d\omega=n!$, we have $\rho=\lim_{n\to  \infty}\left(\frac{\mu_{2n}}{(2n)!}\right)^{1/(2n)}=1$; note that $\mm(z)$ has a pole at $- \ii$. Since $\K{n}[\mm](-\ii/2)=2\,{\ii}^n$, the sum $\sum_{n=1}^\infty|\K{n}[\mm](-\ii/2)|^2$ diverges. Thus, while $\mm(z)$ is analytic on $\mathds{S}(1)$, the chromatic expansions converge uniformly on
$\mathds{S}(1/2-\varepsilon)$.

\subsection{Example 8: the Herron family}
This example is a slight modification of an example from
\cite{HB}. Let the family of orthonormal polynomials be
given by the recursion $L_0(\omega)= 1$, $L_1(\omega)=\omega$,
and $L_{n+1}(\omega)=\omega/(n+1)L_n(\omega)-
n/(n+1)L_{n-1}(\omega).$ Then
\begin{equation*}
\frac{1}{2}\int_{-\infty}^{\infty}L(m,\omega)\;L(n,\omega)\;
\sech\left(\frac{\pi\omega}{2}\right)
\;d\omega=
\delta(m-n).
\end{equation*}
In this case $\mm(z)=\sech z$ and $\K{n}[\mm](z)= (-1)^{n}\sech
z\, \tanh^{n} z$. The recursion coefficients are given by
$\gamma_n=n+1$ for all $n\geq 0$. If $E_n$ are the Euler
numbers, then $\sech z=\sum_{n=0}^\infty E_{2n}\,z^{2n}/(2n)!$,
with the series converging only in the disc of radius $\pi/2$.
In this case $\mm(z)$ is analytic on the strip $\mathds{S}(\pi/2)$, with the
corresponding chromatic expansions converging uniformly on the strip $\mathds{S}(\pi/4-\varepsilon)$.
In this case \\$\mu_{2n}=4^{-n}\pi^{-2n-1}(2n)!(\zeta(2n+1,1/4)-\zeta(2n+1,3/4))$ and 
\[\rho=\lim_{n\to\infty}\left(\frac{1}{2\pi}(\zeta(2n+1,1/4)-\zeta(2n+1,3/4))\right)^{1/(2n)}=2/\pi\]
Thus, $1/\rho=\pi/2$ ($\zeta(x,y)$ is the generalised Riemann zeta function).

\section{Weakly bounded moment functionals}\label{sec5}
\subsection{}

To study the point-wise convergence of chromatic expansions of functions
which are not in $\LB$ we found it necessary to
restrict the class of chromatic moment functionals. The
restricted class is a slight variant of the class introduced in \cite{IG5}, and is still very broad
and contains functionals that correspond to many classical families of orthogonal
polynomials. For technical simplicity in the definition below all conditions
are formulated using a single constant $M$ in all of the bounds.

\begin{definition}[\!\!\cite{IG5}]\label{def-weak} Let \mom\ be any 
moment functional such that for some $\gamma_n>0$ and some reals $\beta_n$ 
\eqref{three-term} holds for the corresponding family of orthonormal polynomials.

\begin{enumerate} \item \mom\ is \emph{weakly bounded}
if there exist some $M \geq 1$ and some $0\leq p<1$ such that for all $n\geq 0$,

\begin{align}\label{one-weak}
&&\frac{1}{M}(n + 1)^p\leq \gamma_n\leq M (n + 1)^p; &&
\frac{|\beta_n|}{ \gamma_n} \leq M. &&
\end{align}
\item \mom\ is \textit{bounded} if it is weakly bounded with $p=0$.
\end{enumerate}
\end{definition}

Functionals in our Examples 1 -- 5 are bounded. 
For bounded moment functionals \mom\ the
corresponding moment distribution function $ \alpha(\omega)$ has a finite
support (see e.g., \cite{Chih}). However, $\mm(t)$ can be of infinite energy (i.e., not
in $L^2$) as is the case in our Example 2. The moment functional in
Example 6 is weakly bounded but not bounded $(p=1/2)$; the
moment functionals in Examples 7 and 8 are not weakly bounded $(p=1)$.
We note that important examples of classical orthogonal
polynomials which correspond to weakly bounded moment
functionals are provided by the following lemma.

\begin{lemma} Let $C,c>0$ and $0< p<1$,  if  $\da=C{\e}^{-c|\omega|^\frac{1}{p}}d\omega$ and $\int_{-\infty}^{\infty}\da=1$, then the corresponding moment functional \mom\ is weakly bounded.
\end{lemma}
\begin{proof}
Theorem 1.3 in \cite{deift} implies that in such a case $\LIM{n}{\gamma_n}/{(n+1)^p}=1/2$; see \cite{JAT} for the details.\\
\end{proof}

Weakly bounded moment functionals satisfy a useful estimation of
the coefficients  in the corresponding equations
\eqref{inverse} and \eqref{direct} relating the chromatic and
the ``standard'' derivatives.

\begin{lemma}[\!\!\cite{IG5}]\label{bounds}
Assume that $\mom$ is weakly bounded. Then the
following inequalities hold for all non-negative integers $k$ and $n$:
\begin{align}
|(\K{n}\circ \dd^k)[\mm](0)|&\leq(M+1)^{2k}k!^p;\label{b-bound}\\
\left|\K{n} \left[\frac{t^k}{k!}\right](0)\right|&\leq \frac{(3M)^{n}}{k!^p}.\label{mono-bound}
\end{align}
\end{lemma}

\begin{proof}
By \eqref{bkn}, it is enough to prove \eqref{b-bound} for all
$n,k$ such that $n\leq k$. We proceed by induction on $k$. If $k=0$ and $n=0$ then $|(\K{n}\circ \dd^k)[\mm](0)|=\mm(0)=1$ and the left side and the right side of \eqref{b-bound} are equal. Assume that the statement holds for all $n\leq k$. Then for all $n\leq k+1$, applying
\eqref{three-term} to $\dd^{k}[\mm](t)$ and then setting $t=0$ we obtain
\begin{equation}\label{auxrec}|(\K{n}\circ \dd^{k+1})[\mm](0)|\leq\! {\gamma_{n}}|(\K{n+1}\circ\dd^{k})[\mm](0)|+
|\beta_n(\K{n}\circ\dd^{k})[\mm](0)|+{\gamma_{n-1}}\,|(\K{n-1}\circ \dd^{k})[\mm](0)|.
\end{equation}
Case 1: If $n=k+1$ then the first two summands on the right hand side are equal to zero and we obtain using the induction hypothesis that 
\begin{equation*}|(\K{k+1}\circ \dd^{k+1})[\mm](0)|\leq{\gamma_{k}}\,|(\K{k}\circ \dd^{k})[\mm](0)|
\leq M(k+1)^p(M+1)^{2k}k!^p< (M+1)^{2k+2}(k+1)!^p.
\end{equation*}
Case 2: If $n=k$ then the first summand on the right is equal to zero and $|\beta_k|\leq M\gamma_k\leq M^2(k+1)^p$; thus, we obtain
\begin{align*}|(\K{k}\circ \dd^{k+1})[\mm](0)|&\leq\!
|\beta_k(\K{k}\circ\dd^{k})[\mm](0)|+{\gamma_{k-1}}\,|(\K{k-1}\circ \dd^{k})[\mm](0)|\\
&\leq\!M^2(k+1)^p(M+1)^{2k}k!^p+Mk^p\,(M+1)^{2k}k!^p\\
&<\!(M+1)^{2k+2}(k+1)!^p
\end{align*}
Case 3: If $n<k$ then $\gamma_n\leq M(n+1)^p\leq M k^p$; $\beta_n\leq M\gamma_n\leq M^2k^p$ and 
$\gamma_{n-1}\leq M(k-1)^p$; thus, we obtain from \eqref{auxrec}
\begin{align*}\label{auxrec}|(\K{n}\circ \dd^{k+1})[\mm](0)|&\leq\! Mk^p(M+1)^{2k}k!^p+
M^2k^p(M+1)^{2k}k!^p+M(k-1)^p(M+1)^{2k}k!^p\\
&<\! (M+1)^{2k+2}(k+1)!^p.
\end{align*}

Similarly, by \eqref{zero-mon}, it is enough to prove
\eqref{mono-bound} for all $k\leq n$. This time we proceed by
induction on $n$. If $n=k=0$ the left side of \eqref{mono-bound} is equal to the right side.  We now use \eqref{three-term} and \eqref{one-weak} to obtain that for all $k\leq n+1$
\begin{align*}
\left|\K{n+1}\left[\frac{t^k}{k!}\right]\right| &\leq \frac{1}{\gamma_n}\left|(\K{n}\circ D)\left[\frac{t^k}{k!}\right]\right|+ \frac{|\beta_n|}{\gamma_n}\left|\K{n}\left[\frac{t^k}{k!}\right]\right|+\frac{\gamma_{n-1}}{\gamma_n}\left|\K{n-1}\left[\frac{t^k}{k!}\right]\right|.
\end{align*}
Thus,
\begin{align*}
\left|\K{n+1}\left[\frac{t^k}{k!}\right]\right|(0) &\leq  \frac{M}{(n+1)^p} \left|\K{n}\left[\frac{t^{k-1}}{(k-1)!}\right]\right|(0)+ M\left|\K{n}\left[\frac{t^{k}}{k!}\right]\right|(0)+ M^2\,
\left|\K{n-1}\left[\frac{t^k}{k!}\right]\right|(0).
\end{align*}
By induction hypothesis and using \eqref{zero-mon} again, we
obtain that for all $k\leq n+1$, we have $k^p\leq(n+1)^p$ and thus
\begin{align*}
\left|\K{n+1}\left[\frac{t^k}{k!}\right](0)\right|\leq\frac{M}{k^p}\,\frac{(3M)^{n}}{(k-1)!^p}+M \frac{(3M)^n}{k!^p}+ M^2 \frac{(3M)^{n-1}}{k!^p}\leq \frac{(3M)^{n+1}}{k!^p}.
\end{align*}
\end{proof}

\begin{corollary}[\!\!\cite{IG5}]\label{poly-bdd} Let $\mom$ be weakly bounded;
then for every fixed $n$
\[
\LIM{k}\left|\frac{(\K{n}\circ \dd^{k})
\left[\mm\right](0)}{k!}\right|^{1/k}=0,
\]
and the convergence is uniform in $n$.
\end{corollary}

\begin{proof}  By \eqref{b-bound},
\begin{equation}\label{poly-bound-eq}
\left|\frac{(\K{n}\circ \dd^k)\left[\mm\right](0)}{k!}\right|^{1/k}\leq
\frac{(M+1)^2}{k!^{(1-p)/k}} <\frac{(M+1)^2{\e}^{1-p}}{k^{1-p}}\to  0.
\end{equation}
\end{proof}
Taking $n=0$ in the above corollary we obtain the following consequence. 
\begin{corollary}[\!\!\cite{IG5}]\label{moments-asymptotics}
Let $\mm(z)$ correspond to a weakly bounded moment functional
\mom; then
\begin{equation}\label{entire}
\LIM{k}\left|\frac{\mu_k}{k!}\right|^{1/k}=
\LIM{k}\left|\frac{\mm^{(k)}(0)}{k!}\right|^{1/k}
=0.\end{equation}
Thus, since \eqref{limsup} is satisfied with $\rho=0$, function $\mm(z)$ which correspond to a weakly bounded moment functional is an entire function.
\end{corollary}
If \eqref{one-weak} holds with $p=1$,
then Lemma~\ref{bounds} implies only that
\begin{equation}\label{p=1}
\left|\frac{(\K{n}\circ\dd^{k})\left[\mm\right](0)}{k!}\right|^{1/k}\leq (M+1)^2
\end{equation}
Examples 7 and 8 show that in such a case the corresponding function
$\mm(z)$ need not be entire. Thus, if the goal is to obtain 
chromatic expansions of entire functions, the upper bound in
\eqref{one-weak} of the definition of a weakly bounded moment
functional is sharp.
Corollary~\ref{moments-asymptotics}  also has the following consequence.
\begin{corollary}[\!\!\cite{IG5}]\label{cite0}
If \mom\ is weakly bounded, then the corresponding family of
polynomials $\PP$ is a complete system in $\LT$.
\end{corollary}
Thus, Corollary \ref{cite0} implies that the Legendre, Chebyshev, Gegenbauer, Jacobi and Hermite families of orthogonal polynomials are complete
in their corresponding spaces $\LT$.

We will make use of the inequalities $(n/e)^n<n!<n(n/e)^n$ which for all $n$ follow from the Stirling estimate for $n!$.
\begin{lemma}\label{bounds11} If $0\leq p<1$ and $\kappa$ is an integer such that $\kappa\geq 1/(1-p)$, then for all $m$,
\begin{equation}\label{sily1}
m!^{1-p}\geq  \lfloor{m}/{\kappa}\rfloor!
\end{equation}

\end{lemma}
\begin{proof}
We have 
$m!^{1-p}\geq m!^{1/\kappa}>
\left({m}/{e}\right)^{{m}/{\kappa}}$, while $\lfloor m/\kappa\rfloor !$ behaves like $(m/(e \kappa))^{m/\kappa}$.
\end{proof}

\begin{lemma}\label{boundsP} Let $p,\kappa$ be as in the previous lemma. Then for every $l>0$ there exist $L$ and $C>0$
such that for all $z\in \Cset$,
\begin{equation}\label{bound-PKnb}
\sum_{m=0}^{\infty}\frac{|l z|^{m}}{m!^{1-p}}\leq\, C\,{\e}^{|Lz|^\kappa}.
\end{equation}
\end{lemma}

\begin{proof}  For $z$ such that $|lz|>1$ we obtain
\begin{align*}
\sum_{m=0}^{\infty}\frac{|l z|^{m}}{m!^{1-p}}\leq 
\sum_{m=0}^{\infty}\frac{|l z|^{m}}{\lfloor m/\kappa\rfloor!}\leq 
\sum_{m=0}^{\infty}\frac{ |l z|^{\kappa\lfloor m/\kappa\rfloor+\kappa-1}}{\lfloor
m/\kappa\rfloor !}
=\kappa\,|l z|^{\kappa-1}
\sum_{j=0}^{\infty}\frac{|l z|^{\kappa j}}{j!}
 =\kappa\,|l z|^{\kappa-1}\,{\e}^{|lz|^\kappa}.
\end{align*}
For $z$ such that $|lz|\leq1$ we obtain
\begin{align*}
\sum_{m=0}^{\infty}\frac{|l z|^{m}}{m!^{1-p}}\leq \sum_{m=0}^{\infty}\frac{|l z|^{m}}{\lfloor m/\kappa\rfloor!}\leq \sum_{m=0}^{\infty}\frac{ |l z|^{\kappa\lfloor m/\kappa\rfloor}}{\lfloor
m/\kappa\rfloor !}
=\kappa \sum_{j=0}^{\infty}\frac{|l z|^{\kappa j}}{j!}
 = \kappa\,{\e}^{|lz|^\kappa}.
\end{align*}
Thus, for all $z$,
\begin{equation*}
\sum_{m=0}^{\infty}\frac{|l z|^{m}}{m!^{1-p}}\leq\, \kappa(|l z|^{\kappa-1}+1)\,{\e}^{|lz|^\kappa}.
\end{equation*}
One can now choose $L> l$ and $C>0$ such that  \eqref{bound-PKnb} holds for all $z\in\Cset$.
\end{proof}

\begin{lemma}\label{bounds1} Let $p,\kappa$ be as in lemma \ref{bounds11}. Then there exists a constant $L$
such that for every $n$ and every $z\in \Cset$,
\begin{equation}\label{bound-Knb}
|\K{n}[\mm](z)|\leq  \frac{|L z|^{n}}{n!^{1-p}}\,{\e}^{|L z|^\kappa}.
\end{equation}
\end{lemma}

\begin{proof}  Using the Taylor series for $\K{n}[\mm](z)$,
\eqref{bkn} and  \eqref{b-bound} and inequality $m!n!\leq (m+n)!$ we obtain that for all $z$,
\begin{align*}
|\K{n}[\mm](z)|&\leq \sum_{r=n}^\infty \frac{|(\K{n}\circ D^r)[\mm](0)|}{r!}\,|z|^r\leq\sum_{m=0}^{\infty}\frac{|(M+1)^2
z|^{n+m}}{(n+m)!^{1-p}}\\ &\leq\ \frac{|(M+1)^2 z|^{n}}{n!^{1-p}}
\sum_{m=0}^{\infty}\frac{|(M+1)^2 z|^{m}}{m!^{1-p}}.
\end{align*}
Using  Lemma \ref{boundsP} we can choose $L\geq (M+1)^2$ such that 
$|\K{n}[\mm](z)|\leq \frac{|L z|^{n}}{n!^{1-p}}\,{\e}^{|L z|^\kappa}.$
\end{proof}
Note that, despite the roughness of the estimates in the above proof, the bound provides an accurate estimate of the type of the growth rate of $\mm(z)$, at least in case of bounded moment functionals ($p=0$ and $\kappa=1$) where $\mm(z)$ is of exponential type, as well as in the case of the Hermite polynomials ($p=1/2$ and $\kappa=2$), in which case $\K{n}[\mm](z)= (-1)^{n}(2^{n}\,n!)^{-1/2}\,z^{n}
{\e}^{-z^2/4}$ and thus for $z={\ii}y$ for a real $y$, we have  $|\K{n}[\mm](z)|=\frac{(q|z|)^n}{(n!)^{1/2}}{\e}^{\left(s|z|\right)^2}$ with $q=2^{-1/2}$ and $s=1/2$. 

\subsection{Point-wise convergence of chromatic expansions}

\begin{lemma} Let $p$ be as in Definition \ref{def-weak} and let $f(z)$ be any function analytic in a disc centred at $u\in\Cset$. Then
$\limsup_{n\to  \infty}\left|\frac{\K{n}[f](u)}{n!^{1-p}}
\right|^{1/n}<\infty$ and if $f(z)$ is entire, then $\lim_{n\to  \infty}\left|\frac{\K{n}[f](u)}{n!^{1-p}}
\right|^{1/n}=0$.
\end{lemma}
\begin{proof} Let $\tau>0$ be any number such that $\limsup_{k\to  \infty}\left|\frac{f^{(k)}(u)}{k!}\right|^{1/k}<\tau$.
Then there exists $B$ such that 
$|f^{(k)}(u)| \leq B k!\tau^k$ for all $k$. 
Using \eqref{direct}
and \eqref{mono-bound} we now obtain
\begin{align*}
|\K{n}[f](u)| &\leq \sum_{k=0}^n\left|\K{n}
\left[\frac{t^k}{k!}\right](0)\right||f^{(k)}(u)|
\leq \sum_{k=0}^n \frac{(3M)^n}{k!^p}Bk!\tau ^k= (3M)^nB\sum_{k=0}^n k!^{1-p}\tau^k
\end{align*}
Since $k\leq n$, we have $k!<k(k/\e)^k\leq n(n/\e)^k$ and we obtain
\begin{align*}
|\K{n}[f](u)| &\leq  (3M)^nBn^{1-p}\sum_{k=0}^n \left(\left(\frac{n}{\e}\right)^{1-p}\tau\right)^k= (3M)^nBn^{1-p}\frac{\left(\left(\frac{n}{\e}\right)^{1-p}\tau\right)^{n+1}-1}{\left(\frac{n}{\e}\right)^{1-p}\tau-1}\\
&< (3M)^n \frac{B n^{1-p}\left(\frac{n}{e}\right)^{1-p}\tau}{\left(\frac{n}{e}\right)^{1-p}\tau-1} \left(\left(\frac{n}{e}\right)^{1-p}\tau\right)^{n}
\end{align*}

Consequently 
\begin{align*}
\left|\frac{\K{n}[f](u)}{n!^{1-p}}\right|^{\frac{1}{n}}  &\leq \frac{3M \left(\frac{B n^{1-p}\left(\frac{n}{e}\right)^{1-p}\tau}{\left(\frac{n}{e}\right)^{1-p}\tau-1}\right)^{\frac{1}{n}} \left(\frac{n}{e}\right)^{1-p}\tau}{\left(\frac{n}{e}\right)^{1-p}}=
3M \left(\frac{B n^{2-2p}\tau}{n^{1-p}-e^{1-p}}\right)^{\frac{1}{n}}\tau
\end{align*}
Our claims now follow from the fact that  $\left(\frac{B n^{2-2p}\tau}{n^{1-p}-e^{1-p}}\right)^{\frac{1}{n}}$ is bounded.
\end{proof}

\begin{theorem}\label{PWT}
Let \mom\ be weakly bounded. 
\begin{enumerate}
\item If a function $f(z)$ is analytic on a domain $G\subseteq\Cset$ and
$u\in G$ then the chromatic expansion $\CE[f,u](z)$ of $f(z)$ converges uniformly to $f(z)$ on
a disc $D\subseteq G$, centered at $u$.

\item In particular, if $f(z)$ is entire, then the chromatic expansion
    $\CE[f,u](z)$ of $f(z)$ converges for every $z\in G$
    and the convergence is uniform on every
    disc around $u$ of finite radius.
\end{enumerate}
\end{theorem}

\begin{proof}
By the previous lemma $\limsup_{n\to \infty}|\K{n}[f](u)/n!^{1-p}|^{1/n}<\infty$. Let $\tau>0$ be an arbitrary number such that $\limsup_{n\to \infty}|\K{n}[f](u)/n!^{1-p}|^{1/n}<\tau$. Then there exists $B_\tau>0$ such that 
$|\K{n}[f](u)|<B_{\tau}\tau^nn!^{1-p}$ for all $n$. Let
$L$ and $\kappa$ be such that \eqref{bound-Knb} holds;
then Lemma \ref{bounds1} implies that for all $n$,
\[
|\K{n}[f](u)\K{n}[\mm](z-u)|^{1/n}<\left(B_{\tau}{\e}^{|L(z-u)|^\kappa}\right)^{1/n}L\tau |z-u|.
\]
Thus, the chromatic series converges uniformly inside every disc
$\mathcal{D}(u,r)\subseteq G$,  centered at $u$, of a radius $r<1/(L\tau)$.
Since by \eqref{insidediff}
\begin{equation*}
\K{m}[\CA[f,u](z)]\big|_{z=u}=\sum_{n=0}^{\infty}
(-1)^n\K{n}[f](u)(\K{m}\circ\K{n})[\mm](0)=\K{m}[f](u),
\end{equation*}
$\CA[f,u](z)$ converges to $f(z)$ on $\mathcal{D}(u,r)$.
\end{proof}

\begin{proposition}\label{exp-type}
Assume that \mom\ is weakly bounded for some $0\leq p <1$ and let $\kappa$ be an integer such that $\kappa\geq 1/(1-p)$ holds.
Assume also that $f(z)$ satisfies that there exists $Q\geq 1$ such that $|\K{n}[f](0)|<Q^n$ for all $n$.
Then there exist $C,L>0$ such that $|f(z)|\leq
C{\e}^{L|z|^\kappa}$. 
\end{proposition}

\begin{proof}
Assume that $|\K{n}[f](0)|<Q^n$; then $|\K{n}[f](0)/n!^{1-p}|^{1/n}\to  0$. Consequently, by Theorem \ref{PWT}  the chromatic expansion of $f(z)$ converges everywhere to $f(z)$ and, using Lemma \ref{boundsP} and estimate \eqref{bound-Knb}  from Lemma \ref{bounds1},  we obtain
\begin{align*}
|f(z)|&\leq \sum_{n=0}^{\infty}|\K{n}[f](0)||\K{n}[\mm](z)|
\leq \sum_{n=0}^{\infty} \frac{|Q\,L_1 z|^{n}}{n!^{1-p}}\,{\e}^{|L_1 z|^\kappa}
\leq C_1{\e}^{|L_2 z|^\kappa}{\e}^{|L_1 z|^\kappa}\leq C{\e}^{L| z|^\kappa}
\end{align*}
for suitably chosen constants. 
\end{proof}

\begin{corollary} Assume that \mom\ is weakly bounded and $\kappa$  an integer such that $\kappa\geq 1/(1-p)$.
If $f(z)\in\LB$ then there exist $C,L>0$ such that $|f(z)|\leq
C{\e}^{L|z|^\kappa}$. 
\end{corollary}

So, while functions $f(z)\in\LB$ are bounded inside the strip $\st{2}$, outside of $\st{2}$, in general, they grow as $C{\e}^{|L z|^\kappa}$.
For bounded moment functionals, such as those
corresponding to the Legendre or the Chebyshev polynomials, we
have $p=0$ and thus $\kappa=1$; consequently, Proposition \ref{exp-type} implies that
in the bounded case functions which are in \LB\ are of exponential type, i.e., $|f(z)|\leq C{\e}^{|L z|}$. For \mom\
corresponding to the Hermite polynomials p=1/2 (see Example 3);
thus, we get that there exists $C,L>0$ such that $|f(z)|\leq
C{\e}^{L|z|^2}$ for all $f\in\LB$. Such a bound is asymptotically tight (modulo the constants $C$ and $L$) because
for the Hermite polynomials $\mm(z)={\e}^{-z^2/4}$; thus for $z=\ii y$ for $y>0$ we have $\mm(z)=\mm(\ii y)={\e}^{y^2/4}={\e}^{|z|^2/4}$. 

More generally, as we have shown, the moment functionals corresponding to exponential weights $\da={\e}^{-c|\omega|^\frac{1}{p}}d\omega$ are weakly bounded. So, if 
\begin{equation*}
f(z)=\int_{-\infty}^{\infty}F(\omega){\e}^{\ii\,\omega\,z}{\e}^{-c|\omega|^\frac{1}{p}}d\omega
\end{equation*}
for some $F(\omega)$ such that 
\begin{equation*}
\int_{-\infty}^{\infty}{|F(\omega)|^2}{\e}^{-c|\omega|^\frac{1}{p}}<\infty,
\end{equation*}
then $F(\omega)\in\LT$ and consequently $f(z)\in\LB$. Thus, $|f(z)|\leq
C{\e}^{L|z|^\kappa}$ for some $C,L>0$ and for $\kappa$ an integer such that $\kappa\geq 1/(1-p)$.
The following conjecture, postulating that the implication in the opposite direction also holds, if true, would be an interesting generalisation of the Paley - Wiener Theorem. 

\begin{conjecture} Assume that  $\mathrm{w}(\omega)$ is a non-negative even weight function and let $f(z)$ be an arbitrary entire function for which  there exists a function $F(\omega)$ such that 
\begin{equation}\label{pw1}
\int_{-\infty}^{\infty}{|F(\omega)|^2}{\mathrm{w}(\omega)}d\omega<\infty
\end{equation}
and such that for all $z\in\Cset$
\begin{equation}\label{pw2}
f(z)=\int_{-\infty}^{\infty}F(\omega)\mathrm{w}(\omega){\e}^{\ii\,\omega\,z}d\omega.
\end{equation}
Then: 
\begin{enumerate} 
\item $f(z)$ satisfies $|f(z)|<C{\e}^{L|z|}$ for some $L,C>0$ just in case the weight function $\mathrm{w}(\omega)$ in equations \eqref{pw1} and \eqref{pw2} can be chosen to be supported in a compact interval;
\item $f(z)$ satisfies $|f(z)|<C{\e}^{L|z|^\kappa }$ for some $L,C>0$ and an integer $\kappa > 1$ just in case $\mathrm{w}(\omega)$ can be chosen  to satisfy
$\mathrm{w}(\omega)\leq {\e}^{-B |\omega|^{-\frac{\kappa}{\kappa -1}}}$ for some $B>0$ and all $\omega\in\Rset$.
\end{enumerate}
\end{conjecture}
Note that, unlike what we have conjectured,  the Paley Wiener Theorem requires $f(t)$ to be in $L^2$; however, for example, while the Bessel function $J_0(t)$ of the first kind is not in $L^2$ but it is of exponential type and it satisfies $J_0(z)=\frac{1}{\pi}\int_{-1}^{1}\frac{1}{\sqrt{1-\omega^2}}{\e}^{\ii\,\omega\,z}d\omega$.

 \subsection{Generalisations of some classical equalities for
the Bessel functions}

Theorem \ref{PWT} generalizes
the classic result that every entire function can be expressed
as a Neumann series of Bessel functions \cite{WAT}, by
replacing the Neumann series with a chromatic expansion that
corresponds to any weakly bounded moment functional. Thus,
many classical results on Bessel functions from \cite{WAT}
immediately follow from Theorem \ref{PWT} and generalise to functions
$\K{n}[\mm](z)$ corresponding to any weakly bounded moment
functional \mom. Below we give a few illustrative examples.

\begin{corollary}\label{eiw}
Let $\PP$ be the orthonormal polynomials associated
with a weakly bounded moment functional \mom; then, by the completeness of  $\PP$ in \LT, Proposition \ref{some} with $\varphi(\omega)\equiv 1$ implies that the equation 
\begin{equation}
{\e}^{{\ii}\omega z}=\sum_{n=0}^{\infty}(-\ii)^{n}
p_{n}(\omega)\,\K{n}[\mm](z)
\end{equation}
holds in the sense of $\LT$ for all $z\in\st{2}$. However, such an equation also holds point-wise for all $z\in\Cset$ and all $\omega$, including $\omega$ outside of the support of $\alpha(\omega)$.
\end{corollary}
\begin{proof}
Follows directly from Theorem \ref{PWT} and
\eqref{iwt}.
\end{proof}
Corollary \ref{eiw} generalises the well known equality for the
Chebyshev polynomials $T_{n}(\omega)$ and the Bessel functions
${\mathrm J}_{n}(z)$, i.e.,
\[
{\e}^{\ii \omega z}={\mathrm J}_0(z) + 2\sum_{n=1}^{\infty}
{\ii}^{n}T_{n}(\omega){\mathrm J}_{n}(z).
\]

In Example 6, \ref{eiw} becomes the equality for the Hermite
polynomials $H_n(\omega)$:
\[
{\e}^{\ii \omega z}=\sum_{n=0}^{\infty}
\frac{H_{n}(\omega)}{n!}\left(\frac{\ii z}{2}\right)^{n}
{\e}^{-\frac{z^2}{4}}.
\]

Using Proposition \ref{unif-con} to expand $\mm(z+u)\in\LB$
into chromatic series around $z=u$, we get that for all
$z,u\in\Cset$
\begin{equation*}\mm(z+u)=\sum_{n=0}^{\infty}(-1)^n
\K{n}[\mm](u)\K{n}[\mm](z),
\end{equation*}
which generalises the equality
\[{\mathrm J}_0(z+u)={\mathrm J}_0(u){\mathrm J}_0(z)+2
\sum_{n=1}^{\infty}(-1)^n{\mathrm J}_n(u){\mathrm J}_n(z).
\]

For the next formula, in order to get a simple form we restrict ourselves to weakly bounded symmetric moment functionals. Thus, $\beta_n=0$ and if we apply Theorem \ref{PWT} to the constant function
$f(z)\equiv 1$, we get that $\K{2n}[f(z)]=\prod_{k=1}^{n}
\frac{\gamma_{2k-2}}{\gamma_{2k-1}}$ and $\K{2n+1}[f(z)]=0$. This implies that 
 the chromatic expansion of $f(z)\equiv 1$ yields that for all weakly bounded moment functionals and 
for all $z\in\Cset$
\begin{equation*}
\mm(z)+\sum_{n=1}^{\infty}(-1)^n \left(\prod_{k=1}^{n}
\frac{\gamma_{2k-2}}{\gamma_{2k-1}}\right)\K{2n}[\mm](z)=1,
\end{equation*}
with $\gamma_n$ the recursion coefficients from \eqref{poly}.
This equality generalises the special case which is the equality for the Bessel functions 
\[
{\mathrm J}_{0}(z)+2\sum_{n=1}^{\infty}
{\mathrm J}_{2n}(z)=1.
\]

\section{Some spaces of almost periodic functions}\label{nonsep}

\subsection{}
In signal processing functions with a finite $L^2$ norm are considered to be signals of finite energy; on the other hand, functions such as 
$f(t)={\e}^{\ii\, \omega\, t}$ have infinite energy but finite power, defined as 
\begin{align}\label{bes}
\|f\|_b=\limsup_{T\to \infty}\left(\frac{1}{2T}\int_{-T}^{T}|f(t)|^2 dt\right)^{1/2}.
\end{align}
The space $B^2$ of the Besicovitch almost periodic functions is the closure of the trigonometric polynomials under the semi-norm $\|f\|_b$ given by \eqref{bes}.

In our setup, functions from $\LL$, thus satisfying $\sum_{n=0}^{\infty}| \K{n}[f](0)|^2<\infty$, can be considered to be signals of finite energy. 
For weakly bounded functionals \mom\ the pure harmonic oscillations $f(t)={\e}^{\ii \omega t}$ do not belong to $\LL$ because $\sum_{n=0}^{\infty}| \K{n}[f](0)|^2=\sum_{n=0}^{\infty}p_n(\omega)^2$
diverges. We now introduce some normed spaces in which pure harmonic oscillations have finite norms; thus, functions from such spaces can be considered to have finite power. 
Note that weakly bounded moment functionals, since $0\leq p<1$, satisfy
\[
\sum_{n=0}^{\infty}\frac{1}{\gamma_n}\geq\sum_{n=0}^{\infty}\frac{1}{M(n+1)^p}=\infty.
\]
We weaken the condition that $1/M(n+1)^p\leq \gamma_n\leq M(n+1)^p$ for some $M$ and $0\leq p<1$ by replacing it with the condition that $\sum_{n=0}^{\infty}\frac{1}{\gamma_n}$ should diverge, but impose some additional conditions on the recursion coefficients, listed below. For simplicity, we consider only symmetric moment functionals. The conditions given below can be weakened, as shown in \cite{GT1}  and as discussed in detail in  \cite{GS1} .  The theory can also be extended to non symmetric functionals which satisfy $\lim_{n\to\infty}|\beta_n|/\gamma_n<2$; see \cite{GT} and \cite{GT2}.

Let $\s{n}$ be the first and $\ds{n}$ the second order forward finite differences of the recursion coefficients:
\begin{align*}
&\s{n}=\g{n+1}-\g{n}; && \ds{n}=\s{n+1}-\s{n}. &
\end{align*}
We consider families of orthonormal polynomials such that the corresponding recursion coefficients $\gamma_n$ satisfy the following conditions. %
\newlist{UR}{enumerate}{3}
\setlist[UR]{label=($\mathcal{C}_{\arabic*})$}
\begin{UR}
\begin{multicols}{2}
\item\label{c1} $\g{n}\rightarrow \infty$;\\
\item\label{c2} $\s{n}\rightarrow 0;$
\end{multicols}
\item \label{c3} There exist $n_0, m_0$ such that  $\g{n+m}>\g{n}$ holds for  all $n\geq n_0$ and all $m\geq m_0$.
\begin{multicols}{2}
\item\label{c4} $\displaystyle{\sum_{j=0}^{\infty}\frac{1}{\gamma_j}}=\infty$; 
\item\label{c5} there exists  $\kappa>1$ such that $\ \displaystyle{\sum_{j=0}^{\infty}\frac{1}{\gamma_j^\kappa}}<\infty$;
\end{multicols}
\begin{multicols}{2}
    \item\label{c6} $\displaystyle{\sum_{n=0}^{\infty}\frac{|\s{n}|}{\gamma_n^2}<\infty;}$
    \item\label{c7} $\displaystyle{\sum_{n=0}^{\infty}\frac{|\ds{n}|}{\g{n}}<\infty.}$
    \end{multicols}
\end{UR} 

Note that if the Hermite polynomials are normalised into a corresponding orthonormal family with respect to the weight $W(x)=\e^{-x^2}/\sqrt{\pi}$, then their recursion coefficients are of the form 
$\g{n}=\frac{1}{\sqrt{2}}\left(n+1\right)^{{1}/{2}}$. 

 While at the first sight these conditions might appear rather complicated, it turns out that the recursion coefficients of many important families of orthonormal polynomials satisfy these conditions, as exemplified by the following lemma.

\begin{lemma}[\cite{JAT}]\label{lemma0} Conditions \ref{c1}-\ref{c7} are satisfied  by the Hermite polynomials, and more generally,  
\begin{enumerate}
\item[(a)] by  families with recursion coefficients of the form $\g{n}=c(n+1)^p$ for any $0<p<1$ and $c>0$;
\item[(b)] by families orthonormal with respect to the exponential weight $W(\omega)=\exp(-c|\omega|^\beta)$ for $\beta>1$ and $c>0$.
\end{enumerate}
\end{lemma}

\begin{theorem}[\cite{JAT}]\label{t1}
Assume that the recursion coefficients $\g{n}$ which correspond to a symmetric positive definite family of orthonormal polynomials $(p_n:n\in\Nset)$ satisfy conditions \ref{c1}-\ref{c7}; then the limit $\lim_{n\rightarrow \infty}\g{n}(p^2_{n}(\omega)+p^2_{n+1}(\omega))$ exists for every $\omega\in\Rset$; more over, for every $B>0$ there exist $m_B$ and $M_B$ such that $0<m_B<M_B<\infty$ and such that  for all $\omega$ which satisfy $|\omega|\leq B$,
\begin{equation}\label{eqbounds}
m_B\leq\lim_{n\rightarrow \infty}\g{n}(p^2_{n}(\omega)+p^2_{n+1}(\omega)) \leq M_B
\end{equation}
with the limit converging uniformly on the set of all $\omega$ such that $|\omega|\leq B$.
\end{theorem}
Note that this is in contrast with the case when the recursion coefficients are bounded; a classical result 
of Nevai implies that in the bounded case the sequence $\g{n}(p^2_{n}(\omega)+p^2_{n+1}(\omega))$ cannot converge; see \cite{PN}.  So it is not surprising that in our case, the slower the $\g{n}$ grow to infinity, the slower the above limit converges.

\begin{corollary}[\cite{JAT}]\label{corollary1}

Let $\gamma_n$ be as in Theorem~\ref{t1}; then the limits below exist and satisfy
\begin{equation}\label{eqsum1}
\lim_{n\rightarrow \infty}
\frac{\sum_{k=0}^np^2_k(\omega)}{\sum_{k=0}^n\frac{1}{\gamma_k}}=\frac{1}{2}\lim_{n\rightarrow \infty}\g{n}(p^2_{n}(\omega)+p^2_{n+1}(\omega))
\end{equation}
and convergence of the two limits is uniform on every compact set. \footnote{Corollary \ref{corollary1} implies our conjecture from \cite{EJA} (under the presence of additional assumptions stated above).}
\end{corollary}

\begin{corollary}[\cite{JAT}]\label{corollary2} If 
 $\gamma_n=c(n+1)^p$ for some $c>0$ and some $p$ such that $0<p<1$, then 
 \[
\displaystyle{
0<\lim_{n\rightarrow\infty}\frac{\sum_{k=0}^np^2_k(\omega)}{(n+1)^{1-p}}<\infty,
}
\]
and convergence of the limit  is uniform on every compact set.
\end{corollary}

\subsection{Spaces of signals of finite power}

\begin{definition} Assume that the recurrence coefficients satisfy conditions \ref{c1}-\ref{c7}.
We denote by $\CC$ the vector space of functions $f(t)\in\Dset$ for which the sequence of corresponding functions $(\beta_n^f(t)\ :\ n\in \Nset)$ defined by 
\begin{align}\label{boundf}
\beta_n^f(t)=\gamma_n(|\K{n}[f](t)|^2+|\K{n+1}[f](t)|^2)
\end{align}
converges uniformly on every compact interval $I\subset \Rset$.
\end{definition}

\begin{lemma}\label{FG} If  $f(t)\in\CC$ then $(\beta_n^f(t)\ :\ n\in \Nset)$ converges to a constant function, $\lim_{n\to \infty}\beta_n^f(t)=L$. Moreover, if we define 
\begin{align}
\nu^{f}_n(t)&=\frac{\sum_{k=0}^{n} |\K{k}[f(t)]|^2}{\sum_{k=0}^{n}\frac{1}{\gamma_k}};
\end{align}
then $\nu^{f}(t)$ converges to constant function $L/2$, uniformly on every compact interval.
\end{lemma}
\begin{proof}
Assume $\beta_n^f(t)\rightarrow l(t)$, uniformly on every compact interval $I\subset \Rset$. 
We now note that our assumptions on $\gamma_n$ imply that 
\begin{align*}
\frac{1}{2}\lim_{n\to\infty}\gamma_n(|\K{n}[f](t)|^2+|\K{n+1}[f](t)|^2)=\lim_{n\to\infty}
\frac{|\K{2n}[f](t)|^2+|\K{2n+1}[f](t)|^2}{\frac{1}{\gamma_{2n}}+\frac{1}{\gamma_{2n+1}}}=\frac{l(t)}{2}.
\end{align*}
For odd $n$, $n=2k+1$, since $\sum_{j=0}^ 1/\gamma_j$ diverges, the (general) Stolz - Ces\`{a}ro Theorem applies and we obtain 
\[ \frac{\sum_{j=0}^{2k+1}|\K{j}[f(t)]|^2}{\sum_{j=0}^{2k+1}\frac{1}{\gamma_j}}=
\frac{\sum_{j=0}^k(|\K{2j}[f(t)]|^2+|\K{2j+1}[f(t)]|^2)}
{\sum_{j=0}^k\left(\frac{1}{\gamma_{2j}}+\frac{1}{\gamma_{2j+1}}\right)}\rightarrow \frac{l(t)}{2}\]
For even $n=2k+2$ we observe that 
\[\frac{\sum_{j=0}^{2k+2}|\K{j}[f(t)]|^2}{\sum_{j=0}^{2k+2}\frac{1}{\gamma_j}}=
\frac{\sum_{j=0}^{2k+1}|\K{j}[f(t)]|^2}{\sum_{j=0}^{2k+1}\frac{1}{\gamma_j}}
\left(1-\frac{\frac{1}{\gamma_{2k+2}}}{\sum_{j=0}^{2k+2}\frac{1}{\gamma_j}}\right)+
\frac{\K{2k+2}[f(t)]|^2}{\sum_{j=0}^{2k+2}\frac{1}{\gamma_j}}
\]

We now prove that $\nu^{f}(t)$ is constant on $\Rset$. 
We observe that  \eqref{C-D}  with $g(x)=f(x)$ yields
\begin{align*}
\left|\frac{d}{d t}\nu_{n}^{f}(t)\right|&\leq \frac{\gamma_{n}\,
(|\K{n+1} [f(t)]\,\overline{\K{n} [f(t)]}|+|\K{n} [f(t)]\,\overline{\K{n+1} [f(t)]}|)}{\sum_{k=0}^{n}\frac{1}{\gamma_k}}\\
&\leq
\frac{\gamma_{n}
(|\K{n} [f(t)]|^2+|\K{n+1} [f(t)]|^2)}{\sum_{k=0}^{n}\frac{1}{\gamma_k}}
\end{align*}
We now apply  a theorem on the exchange of the order of differentiation and summation. Since the numerator  is uniformly convergent on every compact interval it is bounded on every compact interval and since $\sum_{k=0}^{n}{1}/{\gamma_k}$ diverges, the  assumption \eqref{boundf} implies that $(\nu_{n}^{f}(t))^\prime$ converges to $0$ uniformly on I. Thus, for some $L\geq 0$, $\beta_n^f(t)$ and $\nu_n(t)$ converge to constant functions $L$ and $L/2$ respectively.
\end{proof}

\begin{corollary}\label{30}
The following expression,
whose value does not depend on $t$, defines a semi-norm on \CC:
\begin{align}
\Norm{f}=\lim_{n\to\infty}\sqrt{\frac{\sum_{k=0}^{n} |\K{k}[f(t)]|^2}{\sum_{k=0}^{n}\frac{1}{\gamma_k}}}.
\end{align}
Thus, if we let $\CC_0\subset \CC$ to be the set of all $f\in\CC$ such that 
\[ \lim_{n\to\infty}\sqrt{\frac{\sum_{k=0}^{n} |\K{k}[f(t)]|^2}{\sum_{k=0}^{n}\frac{1}{\gamma_k}}}=0,\] 
then $\Norm{f}$ is a norm on the quotient space $\CC_2=\CC/\CC_0$.
\end{corollary}
Note that if $f(t)\in\LL$ then $\sum_{k=0}^{n} |\K{k}[f(t)]|^2$ converges and thus, since 
$\sum_{k=0}^{n}\frac{1}{\gamma_k}$ diverges, $\Norm{f}=0$, which means that $\LL\subseteq\CC_0$.

\begin{lemma}\label{prop1} Let $f,g\in\CC$ and let 
\begin{equation}\label{sigma}
\sigma^{fg}_n(t)=\frac{\sum_{k=0}^{n}
\K{k}[f(t)]{\K{k}[\overline{g(t)}]}}{\sum_{k=0}^{n}\frac{1}{\gamma_k}}.
\end{equation}
If  for some $t=t_0$ the sequence $\sigma^{fg}_n(t_0)$ converges as $n\rightarrow\infty$, then $\sigma^{fg}_n(t)$ converges for all $t$ to a constant function.
\end{lemma}

\begin{proof} 
Using  \eqref{C-D}  we get 
\begin{align*}
\left|\frac{d}{d t}\sigma_{n}^{fg}(t)\right|&\leq\frac{\g{n}\,
(|\K{n+1}_t[f(t)]\,\K{n}_t[\overline{g(t)}]|+|\K{n}_t[f(t)]\,\K{n+1}_t[\overline{g(t)}]|)}{\sum_{k=0}^{n}\frac{1}{\gamma_k}}\\
&\leq
\frac{\g{n}
(|\K{n}_t[f(t)]|^2+|\K{n+1}_t[f(t)]|^2+|\K{n}_t[g(t)]|^2+|\K{n+1}_t[g(t)]|^2)}{2\sum_{k=0}^{n}\frac{1}{\gamma_k}}.
\end{align*}
Since $f,g\in\CC$, the numerator converges uniformly on every finite interval; since the denominator diverges, $\left|\frac{d}{d t}\sigma_{n}^{f}(t)\right|$ converges uniformly to zero on every finite interval. Thus, if $\sigma^{fg}(t_0)$ converges for some $t=t_0$ it converges for all $t$ to a constant function. 
\end{proof}

\begin{theorem}\label{HA}
If the recurrence coefficients satisfy conditions \ref{c2}-\ref{c6}, then the complex exponentials $e_{\omega}(t)={\e}^{\ii \omega t}$ belong to the associated space $\CC_2$ and  for every $e_{\omega}(t)={\e}^{\ii \omega t}$ and $e_{\sigma}(t)={\e}^{\ii \sigma t}$ such that $\omega\neq \sigma$ we have 
\begin{align*}
 \LIM{n}\frac{\sum_{k=0}^{n} \K{k}[{\e}^{\ii \omega t}]  \K{k}[\overline{{\e}^{\ii \sigma t}}] }
{\sum_{j=0}^{n}\frac{1}{\gamma_j}}= 0.
\end{align*}
\end{theorem}
\begin{proof}
Let $\omega> 0$;  since
\[
 |\K{n}[{\e}^{\ii \omega t}]|^2 = |{\ii}^n p_{n}(\omega){\e}^{\ii \omega t}|^2= p_{n}^2(\omega)
\]
and thus 
\begin{equation}{\g{n}( |\K{n}[{\e}^{\ii \omega t}]|^2+ |\K{n+1}[{\e}^{\ii \omega t}]|^2)}=
{\g{n}(p_{n}^2(\omega)+p_{n+1}^2(\omega))},\label{fnorm}
\end{equation}
our Corollary \ref{corollary1} implies that the sequence 
${\g{n}( |\K{n}[{\e}^{\ii \omega t}]|^2+ |\K{n+1}[{\e}^{\ii \omega t}]|^2)}$
converges to a positive finite limit uniformly in $t\in\Rset$ and uniformly in $\omega$ from any compact interval. Thus, ${\e}^{\ii \omega t} \in \CC_2$ and by Lemma \ref{30} we have $0<\|{\e}^{\ii \omega t}\|<\infty$. Let $\omega\neq \sigma$. 
Note that 
\begin{align}\label{KKsin}
\sum_{k=0}^{n} \K{k}[{\e}^{\ii \omega t}]  \K{k}[{\e}^{-\ii \sigma t}] &=\sum_{k=0}^{n}  p_{k}(\omega)p_{k}(\sigma){\e}^{\ii (\omega-\sigma) t}.
\end{align}
By the Christoffel Darboux equality \eqref{CDP} we have
\begin{align}\label{cde}
\sum_{k=0}^{n}  p_{k}(\omega)p_{k}(\sigma)=\frac{\gamma_{n}}{\omega-\sigma}
(p_{n+1}(\omega)p_{n}(\sigma)-p_{n+1}(\sigma)
p_{n}(\omega)).
\end{align}
Thus, 
\begin{align*}
\left|\sum_{k=0}^{n} \K{k}[{\e}^{\ii \omega t}]  \K{k}[{\e}^{-\ii \sigma t}]\right| &=\left|\sum_{k=0}^{n}  p_{k}(\omega)p_{k}(\sigma)\right| \\
&\leq \frac{\gamma_{n}}{\omega-\sigma}
(|p_{n+1}(\omega)p_{n}(\sigma)|+|p_{n+1}(\sigma)
p_{n}(\omega)|)\\
&\leq \frac{\gamma_{n}}{2(\omega-\sigma)}
(p_{n+1}^2(\omega)+p_{n}^2(\sigma)+p_{n+1}^2(\sigma)+
p_{n}^2(\omega)).
\end{align*}
Since $\gamma_{n}
(p_{n+1}^2(\omega)+p_{n}^2(\omega))$ and $\gamma_{n}
(p_{n+1}^2(\sigma)+p_{n}^2(\sigma))$ both converge to a finite limit and $\Sigma_{k=0}^n1/\g{k}$ diverges as $n\rightarrow\infty$,  we obtain
\begin{align*}
 \LIM{n}\frac{\sum_{k=0}^{n} \K{k}[{\e}^{\ii \omega t}]  \K{k}[{\e}^{-\ii \sigma t}] }
{\sum_{k=0}^{n}\frac{1}{\gamma_k}}= 0.
\end{align*}
\end{proof}
Thus, on the vector space of all finite linear combinations of complex exponentials $\LIM{n}\sigma_n^{fg}(t)$ is independent of $t$ and it defines an inner product which makes all complex exponentials of distinct frequencies mutually orthogonal. Such an inner product can now be extended by the same formula to spaces relevant to signal processing, such as spaces of functions of the form 
$f(t)=\sum_{k=0}^\infty q_k{\e}^{\ii \omega_k t}$ where $\sum_{k=0}^\infty |q_k|<\infty$ and for some $B>0$ all $\omega_k$ satisfy $|\omega_k|<B$. Such signals have bounded amplitude and finite bandwidth; they are important, for example, for representation of audio signals. The values of  $\gamma_n(|\K{n}[f(t)]|^2+|\K{n+1}[f(t)]|^2)$ for various $n$ are measures of ``local energy" around an instant $t$ of a signal of infinite total energy (in the usual $L_2$ sense); increasing $n$ reduces such localisation.  Operators $\{\K{k}\}_{k\in\Nset}$ can also be used for frequency estimation of multiple sinusoids in the presence of coloured noise; see \cite{IGF}.

As an example, in the space $\CC_2$ associated with the Hermite polynomials (Example 4) we have $p=1/2$ and a calculation shows that for $\gamma_n=c(n+1)^p$ we have 
\[\LIM{n}\frac{\sum_{k=0}^n\frac{1}{\gamma_n}}{\frac{(n+1)^{1-p}}{c(1-p)}}=1\]
Thus, since for the Hermite polynomials $c=1/\sqrt{2}$, the
corresponding norm is defined by
\[\|f\|^ {H}=\LIM{n}\frac{\sqrt{\sum_{k=0}^{n}|\K{k}[f](t)|^2}}{\sqrt[4]{8n+8}}.\]
Theorem \ref{t1} implies that 
for all $\omega >0 $ functions $f_{\omega}(t)={\e}^{\ii\omega\,t}$ have norm
\begin{equation}\label{HN}
\|f_\omega\|^ {H}=\frac{\e^{\omega^2/2}}{\sqrt[4]{4\pi}}.
\end{equation}

It is an open problem if analogous equalities hold for all bounded moment functionals. However, for particular \mom\  one can still define a norm and a scalar product on the vector space of trigonometric polynomials in which complex exponentials have finite positive norms and in which every two complex exponentials are mutually orthogonal.   For example, for the Chebyshev polynomials (Example 2) one can define 
\begin{align*}&&\|f\|^ {T}=\LIM{n}\sqrt{\frac{\sum_{k=0}^{n}|\K{k}[f](t)|^2}{n+1}};&& 
\langle f,g\rangle^ {T}=\LIM{n}{\frac{\sum_{k=0}^{n}\K{k}[f](t)\overline{\K{k}[g](t)}}{n+1}}.\end{align*}

\begin{proposition}\label{CE}
Let $f_{\omega}(t)={\e}^{\ii \omega t}$;  then $\|f_\omega\|^ {T}=1$ and $\langle f_\omega,f_\sigma\rangle^ {T}=0$ for all $-\pi \leq \omega, \sigma \leq \pi$ such that $\omega\neq \sigma$.
\end{proposition}
\begin{proof}
Using \eqref{iwt} and \eqref{sumsquares},  one can verify that for $-\pi<\omega<\pi$,
\begin{equation*}
\|f_\omega\|^ {T}=\frac{1}{n+1}\sum_{k=0}^{n}\PT{{k}}{\omega}^2 =
\frac{2n+1}{2n+2}+\frac{\sin((2n+1) \arccos\omega)}
{(2n+2)\sqrt{1-\omega^2}}\to  1.
\end{equation*}
Consequently, $\|{f_\omega(t)}\|^T=1$.  Also, from \eqref{iwt} we get
\begin{align*}
\langle f_\omega,f_\sigma\rangle^ {T}&=
\LIM{n}\frac{\sum_{k=0}^{n}\PT{k}{\omega}\PT{k}{\sigma}}{{n+1}}
\end{align*}
Since $\PT{n}{\omega}\leq\sqrt{2}$ on $(-\pi,\pi)$, \eqref{CDP}
implies that for  $\omega\neq\sigma$,
\begin{equation*}
\LIM{n}\frac{\sum_{k=0}^{n}
\PT{k}{\omega}\PT{k}{\sigma}}{n+1}
=\LIM{n}\frac{2}{\pi}\frac{\PT{{n+1}}{\omega}
\PT{{n}}{\sigma}-\PT{{n+1}}{\sigma}\PT{{n}}{\omega}}{(n+1)(\omega-\sigma)}=0.
\end{equation*}
Thus, $\langle f_\omega,f_\sigma\rangle^ {T}=0$. 
\end{proof}

Note that in this case, unlike the case of the family
associated with the Hermite polynomials, the norm of a pure
harmonic oscillation does not depend on its
frequency (but the frequency must belong to the support of the moment distribution function). 

\subsection{A geometric interpretation of chromatic derivatives}
Let $\Langle a_n\Rangle_{n\in\Nset}$ denote a square summable complex sequence,  $\Langle a_n\Rangle_{n\in\Nset}\in l^2$. For every particular value of $t\in \Rset$ the mapping of \LL\
into $l^2$ given by $f\mapsto f^t=\Langle
\K{n}[f](t)\Rangle_{n\in\Nset}$ is a unitary isomorphism which
maps the basis of \LL,  consisting of vectors
$B^0_{k}(t)=(-1)^k\K{k}[\mm(t)]$, into vectors $b_{k}^{t}=
\Langle(-1)^k(\K{n}\circ\K{k})[\mm(t)]\Rangle_{n\in\Nset}$.
Since the scalar product given by \eqref{scl} is independent of $t$, we have 
\begin{align*} \langle b_{k}^{t},b_{m}^{t}\rangle_{l^2}&=\sum_{n=0}^\infty (-1)^k(\K{n}\circ\K{k})[\mm](t)
\overline{(-1)^m(\K{n}\circ\K{m})[\mm](t)}\\
&=\langle(-1)^k\K{k}[\mm](t),\overline{(-1)^m\K{m}[\mm](t)}\rangle_{\!\!\!{\mathcal{M}}}\\
&=\sum_{n=0}^\infty(-1)^k(\K{n}\circ\K{k})[\mm](0) \overline{(-1)^m(\K{n}\circ\K{m})[\mm](0)}\\
&=\delta(n,k)\delta(n,m)=\delta(k,m)
\end{align*}
while \eqref{projf} with $u=0$ implies 
\begin{align*} \langle
f^t,b_{k}^{t}\rangle_{l^2}&=\sum_{n=0}^\infty \K{n}[f](t)\overline{(-1)^k(\K{n}\circ\K{k})[\mm](t)}\\
&=\langle f,(-1)^k\K{k}[\mm]\rangle_{\!\!\!{\mathcal{M}}}\\
&=\sum_{n=0}^\infty\K{n}[f](0)\overline{(-1)^k(\K{n}\circ\K{k})[\mm](0)}\\
&=\K{k}[f](0)
\end{align*}
Thus, since
$\sum_{k=0}^{\infty}|\K{k}[f](0)|^2< \infty$, 
\begin{align*}
\sum_{k=0}^\infty\K{k}[f](0)b_k^t&=\sum\K{k}[f](0)\Langle (-1)^k(\K{n}\circ\K{k})[\mm](t)\Rangle\\
&=\Langle\sum_{k=0}^\infty \K{k}[f](0) (-1)^k(\K{n}\circ\K{k})[\mm](t)\Rangle\\
&=\Langle \K{n}[f](t)\Rangle
\end{align*}
we have
$\sum_{k=0}^{\infty}\K{k}[f](0)\,b_{k}^{t}\in l^2$  for every fixed $t$ and
\begin{align*}
\sum_{k=0}^{\infty}\langle f^t,b_{k}^{t}\rangle\; b_{k}^{t}&=
\sum_{k=0}^{\infty}\K{k}[f](0)\,b_{k}^{t}\\&=
\Big\langle\hspace*{-2mm}\Big\langle\sum_{k=0}^\infty\K{k}[f](0)
(-1)^k(\K{n}\circ\K{k})[\mm(t)]\Big\rangle\hspace*{-2mm}
\Big\rangle_{n\in\Nset}.
\end{align*}
Applying Corollary \ref{postdiff} we obtain
\begin{equation*}
\sum_{k=0}^{\infty}\langle f^t,b_{k}^{t}\rangle\; b_{k}^{t}=
\Langle \K{n}[f](t)\Rangle_{n\in \Nset}=f^t.
\end{equation*}

Thus, while the coordinates of $f^t=\Langle \K{n}[f](t)\Rangle_{n\in
\Nset}$ in the usual basis of $l^2$ vary with $t$, the
coordinates of $f_t$ in the bases $\{b_{k}^{t}\}_{k\in \Nset}$
remain the same as $t$ varies.

We now show that for symmetric $\mom$,  $\{b_{n}^{t}\}_{n\in\Nset}$ is the moving
frame of a helix $H:\Rset\mapsto l^2$.

\begin{lemma}[\!\!\cite{IG5}]\label{continuous}
Assume that $f(t),f^\prime(t)\in\LL$ and let $t\in\Rset$ vary; then $\{f^t\,:\,t\in\Rset\}=
\{\Langle\K{n}[f](t)\Rangle_{n\in\Nset}\,:\,t\in\Rset\}$ is a continuous curve in
$l^2$.
\end{lemma}

\begin{proof} Let $f\in\LL$ and $t, h\in \Rset$; then, using $\left|1-{\e}^{\ii\,\omega\,h}\right|^2=4\sin^2(h\omega/2)$, we obtain
\begin{align*}
\norm{f(t)-f(t+h)}^2&=\int_{-\infty}^{\infty}\left|\FT[f](\omega)\right|^2\left|1-{\e}^{\ii \,h\,\omega}\right|^2\da\\
&= h^2 \int_{-\infty}^{\infty}\left|\omega\FT[f](\omega)\right|^2
\left(\frac{\sin\left(\frac{h\omega}{2}\right)}{\frac{h\omega}{2}}\right)^2\da\\
&\leq h^2\noi{\FT[f^\prime](\omega)}^2\to  0
\end{align*}
\end{proof}

 If $g(t),g^\prime$ and $g^{\prime\prime}(t)$ belong to $\LL$,
then by Lemma \ref{deriv} $\displaystyle{\lim_{|h|\to  0}
\norm{\frac{g(t)-g(t+h)}{h}-g^\prime(t)} = 0}$; thus, the curve
$\vec{g}(t)=\Langle \K{n}[g](t)\Rangle_{n\in\Nset}\in l^2$ is
differentiable, and
$(\vec{g})^{\prime}(t)=\Langle\K{n}[g^{\prime}](t)
\Rangle_{n\in\Nset}$.

Since $\K{k}[\mm](t)\in\LL$ for all $k$, if we let
$\vec{e}_{k+1}(t)=\Langle(\K{k}\circ\K{n})[\mm](t)
\Rangle_{n\in\Nset}$ for $k\geq 0$, then  $\vec{e}_k(t)$ are differentiable
for all $k$. Since $l^2$ is a  complete space and $\vec{e}_1(t)$ is
continuous, $\vec{e}_1(t)$ has an antiderivative $\vec{H}(t)$.
Using \eqref{three-term} for the symmetric case (i.e., when $\beta_n=0$ for all $n$), we obtain
\begin{align*}
\vec{e}_1(t)&=\vec{H}^{\!\!\ \prime}(t);\\\vec{e}_1^{\
\prime}(t)&=\Langle(\dd\circ\K{n})
[\mm](t)\Rangle_{n\in\Nset}=\gamma_0\, \Langle(\K{1}\circ\K{n})
[\mm](t)\Rangle_{n\in\Nset}\\
&=\gamma_0\, \vec{e}_2(t);\label{FS1}\nonumber \\
\vec{e}_k^{\ \prime}(t)&=-\gamma_{k-2}
\Langle(\K{k-2}\circ\K{n})
[\mm](t)\Rangle_{n\in\Nset}+\gamma_{k-1}\Langle(\K{k}\circ\K{n})
[\mm](t)\Rangle_{n\in\Nset}\\
&= - \gamma_{k-2}\,\vec{e}_{k-1}(t) +
\gamma_{k-1}\,\vec{e}_{k+1}(t),\;\;\; \mbox{for $k\geq 2$}.\nonumber
\end{align*}
This means that the curve $\vec{H}(t)$ is a helix in $l^2$
because it has constant curvatures $\kappa_k=\gamma_{k-1}$ for all
$k\geq 1$; the above equations are the corresponding
Frenet--Serret formulas and
$\vec{e}_{k+1}(t)=\Langle(\K{k}\circ\K{n})[\mm](t)\Rangle_{n\in
\Nset}$ for $k\geq 0$ form the orthonormal moving frame of the
helix $\vec{H}(t)$.

\section{Remarks}
The special case of the chromatic derivatives presented in
Example 2 were first introduced in \cite{IG0}; the
corresponding chromatic expansions were subsequently introduced
in \cite{IG00} and first published in \cite{IG1,IG2}. 
These concepts emerged in the course of the
author's design of a pulse width modulation power amplifier.
Subsequently, the research team of the author's startup, \emph{Kromos
Technology Inc.,} extended these notions to various systems
corresponding to several classical families of orthogonal
polynomials \cite{CH,HB}. We also designed and implemented a
channel equaliser and a digital transceiver
(unpublished), based on chromatic expansions. 
In \cite{CNV} chromatic expansions
were related to the work of Papoulis \cite{Pap} and
Vaidyanathan \cite{VA}. In \cite{NIV} and \cite{VIN} the theory
was cast in the framework commonly used in signal processing.
More recently, the theory of chromatic derivatives and expansions was further developed in \cite{ IG5, IG6,WS,Gil,IZ, Zaygen,Sav,Zaygenfunc,SW, Wdisc,Dev,Gilprol,Zayed_Bookchap,horv,Zayfcn }
and successfully applied to several problems in signal processing in \cite{ IGF,WAKI,WIK, WSI, IWK1, IWK2}.

\bibliographystyle{plain}
\bibliography{CD-bib}

\end{document}